\documentclass[a4paper,11pt,english]{amsart}
\usepackage{amsfonts}
\usepackage{amsmath,amsthm}
\usepackage{latexsym}
\usepackage{array}
\usepackage{amssymb}
\usepackage{media9}
\usepackage{graphicx}
\usepackage{enumerate}
\usepackage[english]{babel}

\usepackage{color}
\definecolor{marin}{rgb}   {0.,   0.3,   0.7}
\definecolor{rouge}{rgb}   {0.8,   0.,   0.}
\definecolor{sepia}{rgb}   {0.8,   0.5,   0.}
\usepackage[colorlinks,citecolor=marin,linkcolor=rouge,
            bookmarksopen,
            bookmarksnumbered
           ]{hyperref}

\newcommand{ \red   }{\color{red}}

\newtheorem{lemma}{Lemma}[section]

\newtheorem{theorem}[lemma]{Theorem}
\newtheorem{proposition}[lemma]{Proposition}

\newtheorem{remark}[lemma]{Remark}
\newtheorem{example}[lemma]{Example}

\newtheorem{notation}[lemma]{Notation}
\newtheorem{definition}[lemma]{Definition}
\newtheorem{conclusion}[lemma]{Conclusion}
\numberwithin{equation}{section}
\newcommand{\QED}{\mbox{}\hfill \raisebox{-0.2pt}{\rule{5.6pt}{6pt}\rule{0pt}{0pt}}
          \medskip\par}
\newenvironment{Proof}{\noindent
    \parindent=0pt\abovedisplayskip = 0.5\abovedisplayskip
    \belowdisplayskip=\abovedisplayskip{\bfseries Proof. }}{\QED}

\newcommand{\e}{\ensuremath{\mathrm{e}}}

\newcommand{\Bc}{\mathcal{B}}

\newcommand{\dd}{\mathrm{d}}

\newcommand{\Gc}{\mathcal{G}}
\newcommand{\Hc}{\mathcal{H}}

\newcommand{\N}{\mathbb{N}}

\newcommand{\R}{\mathbb{R}}
\newcommand{\Fc}{\mathcal{F}}

\newcommand{\Lc}{\mathcal{L}}
\newcommand{\T}{\mathbb{T}}
\newcommand{\Z}{\mathbb{Z}}
\newcommand{\Tc}{\mathcal{T}}

\newcommand{\Uc}{\mathcal{U}}
\newcommand{\Vc}{\mathcal{V}}

\newcommand{\Norm}[2]{\|#1\|\left.\vphantom{T_{j_0}^0}\!\!\right._{#2}}

 \author[F. Casas]{Fernando Casas}
\address{IMAC and Departament de Matem\`atiques, Universitat Jaume I, 12071-Castell\'on, Spain.}
\email{Fernando.Casas@uji.es}

\author[N. Crouseilles]{Nicolas Crouseilles}
\address{INRIA-Rennes Bretagne Atlantique, IPSO Project and IRMAR (Universit\'e de Rennes I)}
\email{nicolas.crouseilles@inria.fr}

\author[E. Faou]{Erwan Faou}
\address{INRIA-Rennes Bretagne Atlantique, IPSO Project and IRMAR (Universit\'e de Rennes I)}
\email{Erwan.Faou@inria.fr}

\author[M. Mehrenberger]{Michel~Mehrenberger}
\address{IRMA (Universit\'e de Strasbourg and CNRS) and INRIA-Nancy-Grand Est, TONUS Project}
\email{mehrenbe@math.unistra.fr}

\title[High-order Hamiltonian splitting for Vlasov--Poisson equations]
{High-order Hamiltonian splitting for Vlasov--Poisson equations}

\begin{document}

\begin{abstract}
We consider the Vlasov--Poisson equation in a Hamiltonian framework and derive new time splitting methods based on the decomposition of the Hamiltonian functional between the kinetic and electric energy. Assuming smoothness of the solutions, we study the order conditions of such methods.
It appears that these conditions are of  Runge--Kutta--Nystr\"om type. In the one dimensional case, the order conditions can be further simplified, and efficient methods of order 6 with a reduced number of stages can be constructed. In the general case, high-order methods can also be constructed using explicit computations of commutators.
Numerical results are performed and show the benefit of using high-order splitting schemes in that context. Complete and self-contained proofs of convergence results and rigorous error estimates are also given.
\end{abstract}

\subjclass{  }
\thanks{
}

\maketitle

\section{Introduction}
Frequently, the Vlasov equation is solved numerically with particles methods. Even if
they can reproduce realistic physical phenomena, they are well known to be noisy and
slowly convergent when more particles are considered in the simulation. To remedy
this, the so-called Eulerian methods (which use a grid of the phase space) have known
an important expansion these last decades. Indeed, due to the increase of the machines
performance, the simulation of charged particles by using Vlasov equation can be performed
in realistic configurations. However, these simulations are still computationally
very expensive in high dimensions and a lot has to be done at a more theoretical level
to make simulations faster. For example, the use of high-order methods is classical
when one speaks about space or velocity discretization. However, for the simulation of
Vlasov--Poisson systems, the use of high-order methods in time is not well developed;
generally, only the classical Strang splitting is used and analyzed; see however pioneering
works of \cite{WatanSug2004, Shoucri2005} following \cite{Yoshida1990} or the recent work of
\cite{Schaeffer2010} in the linear case.
We mention also the work \cite{ffm}, which tells us that the increase of order of discretization
in space should be followed with an increase of order in time.

On the other side, a literature exists around the construction of high-order methods
for ODE (see \cite{BlaMo2000, blanes2008, hairer, Sofroniou}).
The main goal of this work is to construct high-order
splitting schemes for the nonlinear Vlasov--Poisson PDE system by the light of these
recent references.

The starting point of our analysis relies on the fact that the Vlasov--Poisson equation is a Hamiltonian PDE
for a Lie--Poisson bracket common to several nonlinear transport equations appearing in fluid dynamics, see for instance \cite{Marsden99} and Section \ref{crochets} below.
Up to phase space discretization, the splitting schemes we construct preserve this structure  and hence are {\em geometric integrators} in the sense of \cite{hairer,leimreich}. Moreover, each block is explicit in time, and can be used to derive high-order methods, taking into account some specific commutator relations.

We consider the following equation Vlasov--Poisson equation
\begin{equation}
\label{vlsp}
\partial_t f + v \cdot \partial_x f - \partial_x \phi(f) \cdot \partial_v f = 0,
\end{equation}
where $f(t,x,v)$ depends on time $t\geq 0$ and the phase space variables $(x,v) \in \T^d \times \R^d$, $d = 1,2,3$,
and where for vectors $(x_1,\ldots, x_d) \in \T^d$ and $(y_1,\ldots, y_d) \in \R^d$,
we set $x\cdot y = x_1 y_1 + \cdots +x_d y_d$ and $|y|^2 := y \cdot y$.
Here, $\T^d$ denotes the $d$ dimensional torus $\R^d \slash (2\pi \Z^d)$
which means that the domain considered is periodic in
space. Note that formally, the same analysis is valid on more general domains, however, we will perform the analysis, in particular the convergence analysis in this simplified framework.  Classically, the variable $x$ corresponds to the
spatial variable whereas $v$ is the velocity variable.

The electric potential $\phi(f)$ solves the Poisson equation
\begin{equation}
\label{eq:elec}
\phi(f)(x)  = - \Delta_{x}^{-1} \left[\int_{\R^d} f(x,v) \dd v - \frac{1}{(2\pi)^d}\int_{\T^d \times \R^d} f(x,v) \dd x \dd v\right],
\end{equation}
where $\Delta_x = \sum_{i = 1}^d \partial_{x_i}^2$ is the Laplace operator in the $x$ variable acting of functions with zero average. The electric field depending only on $x$ is defined as $E(f)  = -\partial_x \phi(f)$.
The energy associated with equations \eqref{vlsp}-\eqref{eq:elec} is
\begin{eqnarray}
\label{scarlatti}
\nonumber \Hc(f) &=& \int_{\T^d \times \R^d} \frac{|v|^2}{2} f(x,v)\dd x \dd v +\int_{\T^d} \frac{1}{2} |E(f)(x)|^2 \dd x.  \\
&=& \Tc(f) + \Uc(f).
\end{eqnarray}
The time discretization methods proposed in this paper are based on this decomposition of the energy.
 Indeed, the solution of the equations associated with $\Tc$ and $\Uc$ can be solved exactly (up to a phase space discretization, for example by interpolation in the framework of semi-Lagrangian methods).
We denote by $\varphi_{\Tc}^t(f)$ and $\varphi_{\Uc}^t(f)$ the flows associated with $\Tc$ and $\Uc$ respectively (we postpone to Section \ref{sec:conv} the precise definition of Hamiltonian flows). The first one corresponds to the equation
$$
\partial_t f+ v \cdot \partial_x f = 0,
$$
for which the solution is written explicitly as
$$
f(t,x,v) = f(0, x - tv,v).
$$
For the flow $\varphi_{\Uc}^t$, we have to solve the equation
\begin{equation}
\label{eq:U}
\partial_t f - \partial_x \phi(f) \cdot \partial_v f = 0,
\end{equation}
for which we verify that the solution is given by
$$
f(t,x,v) = f(0, x, v - tE(f(0))),
$$
where $E(f(0))$ is the value of the electric field 
at time $t = 0$. Indeed, $\phi(f)$ is constant along the solution of \eqref{eq:U}.
Based on these explicit formulae, we will first consider numerical integrators of the form
\begin{equation}   \label{dhom.1}
   \psi_p^{\tau} = \varphi_{\Uc}^{b_{s+1} \tau} \circ \varphi_{\Tc}^{a_s \tau} \circ \varphi_{\Uc}^{b_s \tau} \circ \cdots \circ \varphi_{\Uc}^{b_2 \tau} \circ
       \varphi_{\Tc}^{a_1 \tau} \circ \varphi_{\Uc}^{b_1 \tau},
\end{equation}
where $(a_i)_{i = 1}^s$ and $(b_i)_{i = 1}^{s+1}$ are real coefficients, and
such that the  numerical solution after time $t=\tau$ coincides with the exact solution up to terms of order $\tau^p$, i.e., for a given {\em smooth} Êfunction $f$,
\begin{equation}  \label{dhom.1b}
   \psi_p^{\tau}(f) = \varphi_{\Hc}^{\tau}(f) + \mathcal{O}(\tau^{p+1}),
\end{equation}
where $\varphi_{\mathcal{H}}^{\tau}(f)$ corresponds to the exact flow associated with (\ref{scarlatti}).
We will give a precise definition of smoothness in Section \ref{sec:conv}, and show that this condition ensures the convergence of order $p$ of the numerical method.

As composition of exact flows of Hamiltonians $\Tc$ and $\Uc$, such schemes are (infinite dimensional) Poisson integrators in the sense of \cite[Chapter VII]{hairer}. In particular they preserve the Casimir invariants for the structure for all times (e.g.\ all the $L^p$ norms of $f$). Note that in this work, we do not address the delicate question of
phase space approximation and focus only
on time discretization effects (see \cite{BeMe2008, ChaDeMe2011, ReSo2011}).

To analyze the order of the schemes \eqref{dhom.1}, we will use the Hamiltonian structure of the flows. We will show that they can be expanded in suitable function spaces in terms of commutators, formally reducing the problem to classical settings based on the Baker--Campbell--Haussdorff (BCH) formula and the Lie calculus,
see for instance \cite{hairer,blanes2008}. A rigorous justification will be given in Section \ref{sec:conv}.

In the Vlasov--Poisson case, we will see that the functionals $\Tc$ and $\Uc$ in the decomposition \eqref{scarlatti} satisfy the following formal  relation
\begin{equation}
\label{eq:rknl}
[[[ \Tc, \Uc ], \Uc ], \Uc ] = 0,
\end{equation}
where $[ \cdot, \cdot ]$ is the Poisson
bracket associated with the infinite dimensional Poisson structure (see Section \ref{crochets}).
This property reduces the number of order conditions on the coefficients $(a_i)_{i = 1}^s$ and $(b_i)_{i = 1}^{s+1}$ in formula \eqref{dhom.1}.
The situation is analogous to the case of splitting methods of Runge--Kutta--Nystr\"om (RKN) type for ordinary differential equations (ODE) derived from a
Hamiltonian function, see \cite{hairer,blanes2008}.
In dimension $d = 1$,  the Vlasov--Poisson system
even satisfies the stronger property
$$
[ [\Tc, \Uc ], \Uc] = 2 m \, \Uc,
$$
where $m$ is the total mass of $f$ which is a Casimir invariant, preserved by the exact flow and the splitting methods (\ref{dhom.1}).
This means that we have naturally simpler algebraic order conditions than  those of RKN type for the specific Vlasov--Poisson system in dimension 1.
In any dimension, it also turns out that the exact flow of the Poisson bracket $[ [\Tc, \Uc ], \Uc] $ can be computed up to space discretization.
We will retain these ideas to derive new high-order splitting integrators involving also the flow of this nested Poisson bracket
with optimized coefficients and number of internal steps, in a similar way as in the ODE setting \cite{BlaMo2000,blanes2008}.
The paper is organized as follows:

\begin{itemize}
\item In Section \ref{crochets}, we discuss the Hamiltonian Lie--Poisson structure of the Vlasov--Poisson equation, and give the expressions of some
iterated Poisson brackets. They will form the cornerstone of our analysis.
\item In Section \ref{derivation}, we will first make the link between the standard Lie calculus and the Hamiltonian structure, and then derive high-order splitting methods based on the formula \eqref{dhom.1}. We will then consider generalizations of these methods using explicitly calculable flows of iterated brackets.
\item In Section \ref{sec:num} we give numerical illustrations of the performances of the methods: we mainly exhibit the order of the method, but also address the question of Casimir invariant preservation (e.g.\ the $L^p$ norms),
 regarding the influence of phase space discretizations.
\item Finally, Section \ref{sec:conv} is devoted to the mathematical analysis of the splitting method: we give convergence results in some function spaces. To this aim, we first give a local existence result of the Vlasov--Poisson equation with precise estimates (following in essence \cite{Degond86}), then prove some stability estimates. The results presented in this section can be compared with the one in \cite{Lukas} for the Strang splitting, where however only compactly supported solutions are considered.

\end{itemize}

\section{Hamiltonian structure}
\label{crochets}
\subsection{Poisson brackets}

We define the microcanonical bracket $\{f,g\}$ of two (sufficiently smooth) functions  by the formula
$$
\{f,g\} =\partial_x f \cdot \partial_v g - \partial_v f \cdot \partial_x g.
$$
With this notation, we can write the Vlasov--Poisson equation as
\begin{equation}
\label{arodaky}
\partial_t f - \{h(f), f\} = 0,
\end{equation}
where
$$
h(f)(x,v) = \frac{|v|^2}{2} + \phi(f)(x)
$$
is the microcanonical Hamiltonian associated with $f$. Recall that for a given functional $\Gc(f)$, its Fr\'echet derivative is the distribution $ \frac{\delta \Gc}{\delta f}(f) $ evaluated at the point $f$, being defined by the formula
$$
\Gc(f + \delta f) - \Gc(f) = \int_{\T^d\times\R^d} \frac{\delta \Gc}{\delta f}(f) (x) \, \delta f(x)  \, \dd x \, \dd v + \mathcal{O}(\delta f^2 )
$$
for all smooth variation $\delta f$. In general, a Fr\'echet derivative is an operator acting on $f$, hence a rigorous writing of the previous formula necessitates a loss of derivative in $f$.  We will discuss these issues in Section \ref{sec:conv}.

Considering the two functionals $\Tc(f)$ and $\Uc(f)$ defined in \eqref{scarlatti}, their Fr\'echet derivatives read explicitly
\begin{equation}
\label{bobdylan}
\frac{\delta \Tc}{\delta f}(f) = \frac{|v|^2}{2}
\quad\mbox{and}\quad
\frac{\delta \Uc}{\delta f}(f) = \phi(f)(x),
\end{equation}
 where $\phi(f)$ is given by \eqref{eq:elec}.
Due to the relation $\Hc = \Tc + \Uc$, the Vlasov--Poisson equation can be written as
\begin{equation}
\label{eq:poisson1}
\partial_t f - \{ \frac{\delta \Hc}{\delta f}(f) , f \} = 0,
\end{equation}
which is a Hamiltonian equation for the Poisson structure associated with the following Poisson bracket:
for two functionals $\Hc(f)$ and $\Gc(f)$, we set
\begin{equation}
\label{hgf}
[\Hc,\Gc] = \int_{\T^d \times \R^d} \frac{\delta \Hc}{\delta f}(f)\{ \frac{\delta \Gc}{\delta f}(f), f\} \dd x \dd v= -[\Gc,\Hc],
\end{equation}
where the Fr\'echet functionals are evaluated in $f$.
 Note that the skew-symmetry is obtained using the relation
$$
\{ fg,h\} = f\{g,h\} + g\{f,h\},
$$
for three functions of $(x,v)$ and the fact that the integral in $(x,v)$ of a Poisson bracket of two functions always vanishes. Moreover, this bracket satisfies the Jacobi identity
$$
[\Fc,[ \Gc,\Hc]]Ê+ [\Gc,[ \Hc,\Fc]] + [\Hc,[ \Fc,\Gc]] = 0.
$$
We refer to \cite{Marsden99} for discussions related to this structure. Note that to give a meaning to all the previous expressions, we usually have to assume smoothness for the function $f$ and deal with loss of derivatives, see for instance \eqref{prokofiev} of Section \ref{sec:conv}.

The Hamiltonian--Poisson structure defined above possesses Casimir invariants, meaning quantities preserved for every Hamiltonian system of the form \eqref{arodaky}, and not depending on the specific form of $\Hc$. This is essentially a consequence of the fact that the nonlinear transport equation \eqref{arodaky} involves divergence free vector fields.
Let $\psi:\R \to \R$ be a smooth function, and consider the functional
\begin{equation}
\label{eq:casimirs}
\Psi(f) := \int_{\T^d \times \R^d} \psi(f(x,v)) \dd x \, \dd v.
\end{equation}
Its Fr\'echet derivative is $\frac{\delta \Psi}{\delta f}(f) = \psi'(f)$ and using the definition \eqref{hgf}, we can observe that for {\em all} Hamiltonian functionals $\Hc$, we have
$$
[\Hc,\Psi] =  \int_{\T^d \times \R^d} \frac{\delta \Hc}{\delta f}\{ \psi'(f), f\} \dd x \, \dd v = 0,
$$
owing to the fact that $\{ \psi'(f), f\} = 0$ for all functions $\psi$ and $f$. Hence the functionals \eqref{eq:casimirs} are invariant under any dynamics of the form \eqref{eq:poisson1} (see \eqref{janeweaver} below). They are called Casimir invariants of the Poisson structure. Typical examples are given by the $L^p$ norms of the solution $f$.

\subsection{Relations between $\Tc$ and $\Uc$}

Let us remark that, using \eqref{eq:elec} we have
$$
\int_{\T^d} |E(f)(x)|^2 \dd x
= \int_{\T^d} \phi(f)(x) \int_{\R^d} f(x, v)\dd v\, \dd x,
$$
and hence the potential energy $\Uc$ can be written
$$
\Uc(f) ={\frac{1}{2}}\int_{(\T^d \times \mathbb{R}^d)^2} f(x,v) \, \kappa(x-y) \, f(y,w) \, \dd x \, \dd y \, \dd v \, \dd w,
$$
where
$$
\kappa(x) = \frac{1}{(2\pi)^d}\sum_{k \in \Z^d\slash\{0\}} |k|^{-2} e^{i k \cdot x}
$$
is the kernel of the inverse of the Laplace operator in the $x$ variable.

The aim of this subsection is to prove the following result:

\begin{proposition}
For any smooth function $f$, the functionals $\Tc(f)$ and $\Uc(f)$, satisfy the following relation:
\begin{equation}
\label{tuu}
[[\Tc, \Uc],  \Uc](f) = 2 m(f) \,\Uc(f) + \Vc(f),
\end{equation}
where
\begin{equation}
\label{mf}
m(f) = \frac{1}{(2\pi)^d }\int_{\T^d \times \R^d}  f(x,v) \dd x \, \dd v
\end{equation}
is a constant of motion of \eqref{vlsp},
and
$$
\Vc(f) = {\red -}  \int_{\T^d} \Delta_x  \phi(f)(x) | \partial_x \phi(f)(x)|^2 \dd x,
$$
where $\phi(f)$ is defined in \eqref{eq:elec}.
In dimension $d =1$, we have $\Vc(f) = 0$, and in any dimension $d \geq 1$, the
relation
\begin{equation}
\label{rkn}
[[[\Tc,\Uc],\Uc],\Uc](f) = 0
\end{equation}
holds for all functions $f$.
\end{proposition}

\begin{Proof}
Using \eqref{bobdylan}, we calculate the following
\begin{eqnarray*}
[\Tc,\Uc] &=& - \int_{\T^d\times \R^d} \frac{\delta \Uc}{\delta f}\{ \frac{\delta \Tc}{\delta f}, f\}\dd x \, \dd v \\
&=& - \int_{\T^d\times \R^d}  \phi(f)  \{ \frac{|v|^2}{2} , f\}\dd x \, \dd v\\
&=& \int_{\T^d \times \R^d}  \phi(f)(x) v \cdot \partial_x f(x,v) \dd x \, \dd v.
\end{eqnarray*}
Let us calculate the Fr\'echet derivative of this functional. To this aim, we evaluate this functional at $f + \delta f$,
where $\delta f$ stands for a small perturbation $f$ satisfying  $ \int_{\T^d \times \R^d}  \delta f = 0$.
First, we have
\begin{equation}
\label{derive_phi}
\phi(f + \delta f)(x) = \phi(f)(x) - \Delta_x^{-1} \int_{\R^d} \delta f(x,w) \dd w + \mathcal{O}(\delta f^2).
\end{equation}
Hence, we get
\begin{multline*}
[\Tc,\Uc](f + \delta f)= [\Tc,\Uc]  + \int_{\T^d \times \R^d}  \phi(f)(x) v \cdot \partial_x \delta f(x,v) \dd x \dd v \\
-  \int_{\T^d \times \R^d}  \Big( \Delta_{x}^{-1} \int_{\R^d} \delta f(x,w) \dd w \Big)  v\cdot  \partial_x f(x,v) \dd x \dd v + \mathcal{O}(\delta f^2).
\end{multline*}
We see that, using an integration by parts in $x$, the third term can be written as
$$
- \int_{\T^d \times \R^d} \Big( \int_{\R^d} \delta f(x,v) \dd v \Big) \Delta_x^{-1}( w \cdot \partial_x  f(x,w)) \dd x \dd w.
$$
We deduce that
\begin{eqnarray*}
\frac{\delta [\Tc,\Uc]}{\delta f} (f) &=&  -v \cdot \partial_x \phi (f)(x)  - \int_{\R^d} \Delta_x^{-1} ( w  \cdot \partial_x  f(x,w)) \dd w\\
&=:& v \cdot E(f)(x) + Z(f)(x).
\end{eqnarray*}
Using this relation, we calculate
\begin{eqnarray*}
[[\Tc,\Uc],\Uc]&=& \int_{\T^d \times \R^d} (Z(f)( x) + v \cdot E(f)(x)) \{ \phi(f)(x), f(x,v) \}Ê\dd x \dd v\\
&=& - \int_{\T^d \times \R^d} (Z(f)( x) + v \cdot E(f)(x)) ( E(f)(x) \cdot \partial_v f(x,v))  \dd x \dd v.
\end{eqnarray*}
Now we see that the term involving the function $Z(f)(x)$ vanishes, as the integral of $\partial_v f(x,v)$ in $v \in \R^d$ is equal to $0$.
We can thus write after an integration by parts
\begin{eqnarray*}
[[\Tc,\Uc],\Uc]
&=&  \int_{\T^d \times \R^d} f(x,v) |E(f)(x)|^2 \dd x \dd v.
\end{eqnarray*}
In other words, we get
\begin{equation}
\label{TUU}
[[\Tc,\Uc],\Uc]  = \int_{\T^d} \rho(f)(x) |ÊE(f)(x)|^2 \dd x,\quad \mbox{with}\quad \rho(f)(x) = \int_{\R^d} f(x,v) \dd v.
\end{equation}
But we have with \eqref{mf} and \eqref{eq:elec}
$$
\rho(f)(x)  = m(f) - \Delta_x  \phi(f)(x),
$$
and this yields \eqref{tuu}.
In dimension $d=1$, we can further write that
$$
\Delta_x  \phi(f)(x) | \partial_x \phi(f)(x)|^2 = \frac{1}{3}\partial_x ( \partial_x \phi(f)(x))^3
$$
and conclude that $\Vc$ is identically equal to $0$. In any dimension $d$, as the Fr\'echet derivatives of $\Uc$ and $\Vc$ depend only on $x$, we automatically obtain \eqref{rkn}.
\end{Proof}

\subsection{Flow of the Hamiltonian  $[[\Tc,\Uc],\Uc]$}
\label{ronaldo}

As mentioned above, the Fr\'echet derivative of the Hamiltonian $[[\Tc,\Uc], \Uc]$ only depends on $x$. Hence its exact flow can be calculated explicitly, making it possible to be included in the splitting methods blocks in any dimension. The situation is completely analogous to
Hamiltonian systems in classical mechanics when the kinetic energy is quadratic in momenta, see e.g. \cite{blanes2008,mclachlan,omelyan}.

 From the expression \eqref{TUU} of the Poisson bracket $[[\Tc,\Uc],\Uc]$,
we can calculate its Fr\'echet derivative.

\begin{proposition}
\label{propK}
For any smooth function $f$ and using the notations introduced above, we have
$$
\frac{\delta [[\Tc,\Uc],\Uc]}{\delta f} = K(x, f),
$$
where $K$ satisfies
\begin{equation}
\label{eqonK_lap}
-\Delta_x K = -2m\Delta_x \phi - 2\sum_{i, j = 1}^d \left( \partial_{x_i}\partial_{x_j}\phi  \right)^2+ 2(\Delta_x \phi)^2.
\end{equation}
Denoting by $E_{x_j}$, $j = 1,\cdots, d$ the components of the electric vector fields $E$, we get
in the case $d=2$
\begin{equation}
\label{jacE}
-\Delta_x K = -2m\Delta_x \phi - 4 \left( \partial_{x_1}E_{x_2} \right) \left( \partial_{x_2}E_{x_1} \right) + 4 \left( \partial_{x_1}E_{x_1}\right)\left(\partial_{x_2}E_{x_2}  \right),
\end{equation}
and in the case $d=1$, $\partial_x K = -2mE$.
\end{proposition}

\begin{Proof}
Let us calculate  $[[\Tc,\Uc],\Uc](f+\delta f)$
\begin{align*}
[[\Tc,\Uc],\Uc]&(f+\delta f) = \int \rho(f+\delta f) |E(f+\delta f)|^2 \dd x \\
&= \int \rho(f+\delta f) |\partial_x \phi(f+\delta f)|^2 \dd x  \\
&=  \int \rho(f+\delta f) \left|\partial_x \phi(f) -\partial_x \Delta_x^{-1}\int \delta f \dd w  \right|^2 +\mathcal{O}(\delta f^2)\\
&=  \int \rho(f+\delta f)\left[|E(f)|^2 -2\partial_x\phi\cdot \partial_x \Delta_x^{-1}\int \delta f\dd w\right]\dd x +\mathcal{O}(\delta f^2) ,
\end{align*}
where we used \eqref{derive_phi}. Hence we have
\begin{align*}
[[\Tc,\Uc],\Uc]&(f+\delta f)  - [[\Tc, \Uc], \Uc](f)\\
&= \int |E|^2 \rho( \delta f) \dd x -2\int \rho(f) \left[\partial_x \phi\cdot \partial_x\Delta_x^{-1}\int \delta f\dd w\right]\dd x+\mathcal{O}(\delta f^2)\nonumber\\
&= \int |E|^2 \int \delta f \dd w \dd x - 2\int \partial_x\cdot(\rho(f) E) \Delta_x^{-1}\int \delta f\dd w\dd x+\mathcal{O}(\delta f^2)\nonumber\\
&=\int |E|^2 \int \delta f \dd w \dd x - 2\int \Delta_x^{-1}\left[\partial_x\cdot(\rho E)\right] \int \delta f\dd w \dd x+\mathcal{O}(\delta f^2).
\end{align*}
We deduce, that
\begin{equation}
\label{eqonK_general}
K(x,f) := \frac{\delta [[\Tc,\Uc],\Uc]}{\delta f} (f) =  |E|^2 - 2 \Delta_x^{-1} \mathrm{div} ( \rho E).
\end{equation}
Let us consider its Laplacian of $K$
\begin{eqnarray*}
-\Delta_x K &=&  -2\sum_{i,j = 1}^d\left( \partial_{x_i}E_{x_j} \right)^2-2\sum_{i,j = 1}^dE_{x_j}\partial_{x_i}^2E_{x_j}+2\mathrm{div} ( \rho E)\\
&=&  -2\sum_{i,j = 1}^d\left( \partial_{x_i}\partial_{x_j}\phi \right)^2-2\sum_{i,j = 1}^d\partial_{x_j}\phi \, \partial_{x_i}^2\partial_{x_j}\phi\\
& & -2\sum_{j = 1}^d\partial_{x_j}\rho \, \partial_{x_j}\phi-2\rho\sum_{i = 1}^d\partial_{x_i}^2\phi,
\end{eqnarray*}
where $E_{x_j}$, $j = 1,\cdots, d$, denote the components of the electric vector fields $E$.
Using that $-\-\Delta_x \phi = \rho-m$, with $m$ independent of $x_i$, we obtain
\begin{equation}
\label{gugliemi}
-\Delta_x K = -2m\Delta_x \phi - 2\sum_{i, j = 1}^d \left( \partial_{x_i}\partial_{x_j}\phi  \right)^2+ 2(\Delta_x \phi)^2.
\end{equation}
In dimension $d=2$, we get
\begin{eqnarray*}
-\Delta_x K &=&  -2m\Delta_x \phi - 4 \left( \partial_{x_1}\partial_{x_2}\phi  \right)^2 + 4 \left( \partial^2_{x_1}\phi\right)\left(\partial^2_{x_2}\phi  \right), \nonumber\\
&=& -2m\Delta_x \phi - 4 \left( \partial_{x_1}E_{x_2} \right) \left( \partial_{x_2}E_{x_1} \right) + 4 \left( \partial_{x_1}E_{x_1}\right)\left(\partial_{x_2}E_{x_2}  \right), \nonumber\\
\end{eqnarray*}
which can be solved, the right hand side being of zero average. Let us remark
that the two last term are nothing but the determinant of the Jacobian matrix of  the electric field $E$.

In dimension $d=1$, we have from \eqref{eqonK_general}
\begin{eqnarray*}
\partial_x K &=& 2E \cdot \partial_x E  - 2 \partial_x  ( \Delta_x^{-1} \mathrm{div} (\rho E))\\
&=&  (-2\partial_{xx} \phi - 2 \rho) E \\
&=& - 2 m E
\end{eqnarray*}
as expected.
\end{Proof}

With the previous notations and Proposition \ref{propK},
the equation associated with the Hamiltonian $[[\Tc,\Uc],\Uc]$ is given by
\begin{equation}
\label{dorian}
\partial_t f - \left\{ K, f \right\} = \partial_t f - \partial_x K\cdot \partial_v f = 0.
\end{equation}
Hence the flow associated with $[[\Tc,\Uc],\Uc]$ is explicit and given by
\begin{equation}
\label{dorian2}
\varphi_{[[\Tc,\Uc],\Uc]}^t(f)(x,v) = f(0, x,v + t \partial_x K(x,f(0))),
\end{equation}
because, $K$ depending only on $x$ and integrals of $f$ in $v$, it is constant in the evolution of the flux associated with $[[\Tc,\Uc],\Uc]$.

In dimension $2$ and $3$,  $K$ (and then $\partial_x K$) can be easily computed
in Fourier space by solving \eqref{eqonK_lap}: in particular, the computational cost of a term
of the form $\varphi^t_{\Uc + \gamma [[\Tc,\Uc],\Uc]}$ is essentially the same for $\gamma = 0$ (standard splitting)
as for $\gamma \neq 0$.


\section{Derivation of high-order methods}
\label{derivation}

The splitting methods \eqref{dhom.1} are composition of exact flows of Hamiltonian equations of the form \eqref{arodaky}.
To analyze their orders of approximation, we will use the algebraic structure of the Vlasov--Poisson equation.
For a Hamiltonian equation of the form \eqref{eq:poisson1}, let us define
$$
\mathrm{ad}_{\Hc}f = \{ \frac{\delta \Hc}{\delta f}(f) , f\}.
$$
This notation is justified by the fact that
the equation \eqref{vlsp} is equivalent to
\begin{equation}
\label{janeweaver}
\forall\, \Gc, \quad \frac{\dd}{\dd t}\Gc(f) = [\Hc,\Gc](f) = -\int_{\R^dÊ\times \T^d} \frac{\delta G}{\delta f} \mathrm{ad}_{\Hc}f \, \dd x \, \dd v,
\end{equation}
where $\Gc$ here are functionals acting on some function space. We will not discuss here the mathematical validity of such an equivalence, but taking for instance $\Gc(f)$ as a norm of a function space (see Section \ref{sec:conv}), we can prove that the solution $f(t)$ admits a formal expansion of the form
\begin{equation}
\label{eq:exp1}
f(t) = \sum_{k \geq 0}\frac{ t^k}{k!} \mathrm{ad}^k_{\Hc} f_0 = \exp( t \, \mathrm{ad}_{\Hc}) f_0.
\end{equation}
By using similar expansions for the flows we have
$$
\varphi_\Tc^t (f) =  \sum_{k \geq 0}\frac{ t^k}{k!} \mathrm{ad}^k_{\Tc} f  \quad \mbox{and}\quad  \varphi_\Uc^t(f) = \sum_{k \geq 0}\frac{ t^k}{k!} \mathrm{ad}^k_{\Uc} f.
$$
By using \eqref{janeweaver} and the Jacobi identity, we see that the following relation
$$
[\mathrm{ad}_{\Tc},  \mathrm{ad}_{\Uc}] := \mathrm{ad}_{\Tc} \circ \mathrm{ad}_{\Uc} - \mathrm{ad}_{\Uc} \circ \mathrm{ad}_{\Tc}   = \mathrm{ad}_{[\Tc,\Uc]}
$$
 holds true. We deduce that
the classical calculus of Lie derivatives also applies to our case.

Using this identification, (see also \cite{blanes2008}) we can write formally the exact flows
as
$$
\varphi_\Tc^{t} =: \e^{t T}, \,   \varphi_{\Uc}^{t} =:\e^{t U} ,\quad\mbox{and}\quad \varphi_\Hc^t = \e^{t(T+U)},
$$
with the operators $T$ and $U$ satisfying the relation $[[T,U],U] = 2m U$ in dimension $1$, where $m$ is a constant, or the
RKN-type relation $[[[T,U],U],U] = 0$. To derive splitting methods in dimension $d \geq 2$, we will also consider numerical schemes containing blocks based on the exact computation of the flow associated with the Hamiltonian $[[T,U],U]$. In this section we will concentrate on the derivation of high-order splitting methods of the form (\ref{dhom.1}) satisfying these formal relations.

Scheme (\ref{dhom.1}) is at least
 of order $1$
for the problem (\ref{vlsp}) if and only if the coefficients $a_i$, $b_i$ satisfy the consistency condition
\begin{equation}   \label{consist}
   \sum_{i=1}^{s} a_i = 1, \qquad \sum_{i=1}^{s+1} b_i = 1.
\end{equation}
We are mainly interested in symmetric compositions, that is, integrators such that $a_{s+1-i} = a_i$, $b_{s+2-i} = b_i$, so that
\begin{equation}   \label{dhom.2}
   \psi_p^{\tau} = \e^{b_1 \tau U} \, \e^{a_1 \tau T} \, \e^{b_2 \tau U} \, \cdots \, \e^{b_2 \tau U} \,    \e^{a_1 \tau T} \,    \e^{b_{1} \tau U}.
\end{equation}
In that case,
they are of even order. In particular, a symmetric method verifying (\ref{consist}) is at least of order $2$ \cite{hairer}.
Notice that the number of flows in the splitting method (\ref{dhom.1}) or (\ref{dhom.2}) is $\sigma \equiv 2s+1$, but
the last flow can be concatenated with the first one at the next step in the integration
process, so that the number of flows $\varphi_U^{\tau}$ and $\varphi_T^{\tau}$ per step is precisely $s$.

Restriction (\ref{dhom.1b}) imposes a set of constraints the coefficients $a_i$, $b_i$ in the composition (\ref{dhom.2}) have to satisfy. These are the so-called order conditions of the splitting method and a number of procedures can be applied to obtain them \cite{hairer}. One of them consists in applying recursively
the Baker--Campbell--Hausdorff (BCH) formula in the formal factorization (\ref{dhom.2}). When this done, we can
 express $\psi_p^{\tau}$ as the formal exponential of only one operator
\begin{equation}  \label{asexp}
   \psi_p^{\tau} = \e^{\tau(T+U  + R(\tau))},
\end{equation}
where
\begin{multline}  \label{erre}
  R(\tau)  =  \tau p_{21} [T,U] + \tau^2 ( p_{31} [[T,U],T] +     p_{32} [[T,U],U]) + \\
    \tau^3 ( p_{41} [[[T,U],T],T] + p_{42} [[[T,U],U],T] + p_{43} [[[T,U],U],U] ) + \mathcal{O}(\tau^4),
\end{multline}
and $p_{ij}$ are polynomials in the parameters $a_i$, $b_i$. Here we assume that the coefficients satisfy (\ref{consist}).

The integrator is of order $p$ if $R(\tau) = \mathcal{O}(\tau^{p})$, and thus
the order conditions are $p_{21} = p_{31} = p_{32} = \cdots = 0$ up to the order considered.  For a symmetric scheme one has $R(-\tau) = R(\tau)$,
so that $R(\tau)$ only involves even powers of $\tau$. In consequence, $p_{21} = p_{41} = p_{42} =  \cdots = p_{2n,k} = 0$ automatically in (\ref{erre}) and we have only to impose $p_{31}=p_{3,2}= \cdots = p_{2n+1,k} = 0$.
The total number of order conditions can be determined by computing the dimension of the subspaces spanned by the
 $k$-nested commutators involving $T$ and $U$ for $k=3, 5, \ldots$, see \cite{mclachlan}.

For the problem at hand  $ [[[T,U],U],U] = 0$
identically, and this introduces additional simplifications due to the  linear dependencies appearing at higher order terms in $R(\tau)$. The number of
order conditions is thus correspondingly reduced. In Table  \ref{tab.rkn.1} we have collected this number  for  symmetric methods of order $p=2,4,6,8$
(line $d>1$). Thus, a symmetric 6th-order scheme within
this class requires solving $8$ order conditions (the two consistency conditions (\ref{consist}) plus $2$ conditions at order $4$ plus $4$ conditions at order $6$), so that
the scheme (\ref{dhom.2}) requires at least $15$ exponentials. In fact, it is a common practice to consider more exponentials than
strictly necessary and use the free parameters introduced in that way to minimize error terms. In particular, in \cite{BlaMo2000} a $6$th-order splitting method
involving $23$ exponentials ($11$ stages) was designed which has been shown to be very efficient for a number of problems, including
Vlasov--Poisson systems \cite{jcp14aho}.

\begin{table}
\begin{center}
  \begin{tabular}{|c|cccc|} \hline
  Order $p$ &  2 & 4 &  6 &   8    \\ \hline\hline
  $d>1$   & 2  & 4  & 8 &  18     \\  \hline
  $d=1$  & 2 & 4 &  8 &  16     \\
 \hline
\end{tabular}
\end{center}
  \caption{Numbers of independent order conditions to achieve order $p$ required by
symmetric splitting methods when $ [[[T,U],U],U] = 0$ ($d>1$), and when  $[[T,U],U] = 2 m  U$, with $m$ constant ($d=1$).}
\label{tab.rkn.1}
\end{table}

We have shown in the subsection \ref{ronaldo} that, besides the flow corresponding to $\mathcal{U}$, the flow associated to $[[\Tc,\Uc],\Uc]$ can also be
explicitly computed in a similar way as $\varphi_{\Uc}^{\tau}$. Moreover, since $[[[\Tc,\Uc],\Uc],\Uc]=0$,
both flows commute so that we can consider a composition
(\ref{dhom.1}) with the flow $\varphi_{\Uc}^{b_i \tau}$ replaced by $\varphi^{\tau}_{b_i \Uc + c_i \tau^2 [[\Tc,\Uc],\Uc]}$. Equivalently, in the composition
(\ref{dhom.2}) we replace $\e^{b_i \tau U}$ by $\e^{\tau C_i}$, where $C_i \equiv b_i U + c_i \tau^2 [[T,U],U]$:
\begin{equation}  \label{modifpot}
   \psi_p^{\tau} = \e^{\tau C_1} \, \e^{ a_1 \tau T} \, \e^{\tau C_2} \, \cdots \, \e^{\tau C_2} \, \e^{ a_1 \tau T} \,  \e^{\tau C_1}.
\end{equation}
In that case the order conditions to achieve order 6 are explicitly
\begin{equation}   \label{orcon2}
\aligned
  & \sum_{i=1}^{s+1} b_i \, \Big( \sum_{j=1}^i a_j \Big)^2 = \frac{1}{3};  \qquad\qquad
  \sum_{i=1}^{s+1} a_i \, \Big( \sum_{j=i}^{s+1} b_j \Big)^2 - 2 \sum_{i=1}^{s+1} c_i = \frac{1}{3};  \\
 & \sum_{i=1}^{s+1} b_i \, \Big( \sum_{j=1}^i a_j \Big)^4 = \frac{1}{5};  \qquad\qquad
   \sum_{i=2}^{s+1} b_i \, \bigg( \sum_{j=1}^{i-1} b_j \, \Big( \sum_{k=j+1}^i a_k \Big)^3 \, \bigg) = \frac{6}{5!}; \\
 & \sum_{i=2}^{s+1} a_i \bigg( 2 \sum_{j=1}^{i-1} a_j \Big( \sum_{k=j}^{i-1} c_k \Big) + \sum_{j=1}^{i-2} a_j \sum_{k=j+1}^{i-1} a_k \Big(
   \sum_{\ell=j}^{k-1} b_{\ell} \Big) \Big( \sum_{m=k}^{i-1} b_m \Big) \bigg) = \frac{1}{5!}; \\
 & 2 \sum_{i=2}^{s+1} a_i \Big( b_i \sum_{j=1}^{i-1} c_j + c_i \sum_{j=1}^{i-1} b_j \Big) +
  \sum_{i=2}^s a_i \Bigg(  2 \Big( \sum_{j=i+1}^{s+1} b_j \Big)   \Big(  \sum_{k=1}^{i-1} c_k \Big)   \\
 &  \;\; + 2 \Big( \sum_{j=i+1}^{s+1} c_j \Big)   \Big(  \sum_{k=1}^{i-1} b_k \Big) +
  \Big( \sum_{j=1}^{i-1} b_j \Big) \sum_{k=i+1}^{s+1} a_k  \Big( \sum_{\ell = i}^{k-1} b_{\ell} \Big) \Big( \sum_{m=k}^{s+1} b_m \Big) \Bigg) = \frac{1}{5!},
\endaligned
\end{equation}
together with the consistency conditions (\ref{consist}). Here $a_{s+1} = 0$, $a_{s+1-i} = a_i$, $b_{s+2-i} = b_i$. The two equations in the first line of
(\ref{orcon2}), together with (\ref{consist}), lead to a method of order four.
With the inclusion of $C_i$ in the scheme,
the number of exponentials can be significantly reduced (one has more parameters available to satisfy the order conditions):
the minimum number of exponentials required by the symmetric
method (\ref{modifpot}) to achieve order $6$ is $\sigma = 9$ instead of $15$
for  scheme (\ref{dhom.2}).  There are  several other systems
where the evaluation of the flow associated with  $[[T,U],U]$ is not substantially more expensive in terms of computational cost than the
evaluation of $\e^{\tau U}$, and thus schemes of the form (\ref{modifpot})
have been widely analyzed and several efficient integrators can be found in the literature \cite{blanes2008,omelyan}.

We have considered compositions of the form (\ref{modifpot}) with $\sigma = 9, 11$, and $13$ exponentials. When
$\sigma=9$ there is only one real solution of equations (\ref{orcon2}).
More efficient schemes are obtained by using more exponentials: the corresponding free parameters can be used to optimize the scheme (for instance,
by annihilating higher order terms in $R(\tau)$, reducing the norm of the main error terms, etc.). In Table \ref{tab.rkn.3} we collect the coefficients
of the best methods we have found. The most efficient one (see Section \ref{sec:num}) corresponds to $\sigma = 13$. In this case the two free
parameters have been chosen to vanish the coefficient multiplying the commutator
$[T,[T,[T,[T,[T,[T,U]]]]]]$
 at order 7 and such that $b_1 = b_2$.
This procedure usually leads to very efficient schemes, as shown in \cite{mclachlan02,sisc06}. The scheme reads
\begin{equation}   \label{sche.1}
   \psi_6^{\tau} = \e^{\tau C_1} \, \e^{ a_1 \tau T} \, \e^{\tau C_2} \, \e^{ a_2 \tau T} \, \e^{\tau C_3} \, \e^{ a_3 \tau T} \, \e^{\tau C_4} \, \e^{ a_3 \tau T} \,
      \e^{\tau C_3} \, \e^{ a_2 \tau T} \, \e^{\tau C_2} \, \e^{ a_1 \tau T} \, \e^{\tau C_1}
\end{equation}
and its coefficients are collected in Table \ref{tab.rkn.2}. Notice that all $b_i$ coefficients are positive and only one $a_i$ is negative. All methods from
Table \ref{tab.rkn.3} will be tested and compared in Section \ref{sec:num}.

\begin{table}[ht]
\centering
\begin{tabular}{cl}
\noalign{\smallskip}\hline\noalign{\smallskip}
\begin{tabular}{c}
$\psi_6^{\tau}$ of the form \eqref{modifpot} \\ with $\sigma=9$, $ \;\; d>1$
\end{tabular}
&  \quad
 \begin{tabular}{rcl}
$a_1$ &=&  $1.079852426382430882456991$    \\
$a_2$ &=&  $-0.579852426382430882456991$ \\
$b_1$ &=&  $0.359950808794143627485664$  \\
$b_2$ &=&  $-0.1437147273026540434771131$  \\
$b_3$ &=&  $0.567527837017020831982899$  \\
$c_1$ &=& $0$ \\
$c_2$ &=& $-0.0139652542242388403673$ \\
$c_3$ &=& $-0.039247029382345626020$ \\
\end{tabular}  \\
\noalign{\smallskip}\hline\hline\noalign{\smallskip}
\begin{tabular}{c}
$\psi_6^{\tau}$ of the form \eqref{modifpot} \\ with $\sigma=11$, $\;\; d>1$
\end{tabular}
 & \quad
 \begin{tabular}{rcl}
$a_1$ &=&  $a_2$    \\
$a_2$ &=& $0.303629319055488881944104$ \\
$a_3$ &=&  $-0.2145172762219555277764167$ \\
$b_1$ &=& $0.086971698963920047813358$  \\
$b_2$ &=&    $0.560744966588102145251453$  \\
$b_3$ &=&    $-0.1477166655520221930648117$  \\
$c_1$ & = &  $-1.98364114652831655458915 \cdot 10^{-6}$ \\
$c_2$ & = &  $0.00553752115152236516667268$ \\
$c_3$ & = & $0.00284218110811634663914191$ \\
   \end{tabular} \\
\noalign{\smallskip}\hline\hline\noalign{\smallskip}
\begin{tabular}{c}
$\psi_6^{\tau}$ of the form \eqref{modifpot} \\ with $\sigma=13$, $\;\; d>1$
\end{tabular}
 & \quad
 \begin{tabular}{rcl}
$a_1$ &=&  $0.270101518812605621575254$    \\
$a_2$ &=& $-0.108612186368692920020654$ \\
$a_3$ &=&  $0.338510667556087298445400$ \\
$b_1$ &=& $b_2$  \\
$b_2$ &=&    $0.048233230175303256742758$  \\
$b_3$ &=&    $0.236139260374249444475399$  \\
$b_4$ & =&   $0.334788558550288084078170$ \\
$c_1$ & = &  $0.000256656790401210726353$ \\
$c_2$ & = &  $0.000943977158092759357851$ \\
$c_3$ & = & $-0.002494619878121813220455$ \\
$c_4$ &=&   $-0.002670269183371982607658$
\end{tabular}  \\
\noalign{\smallskip}\hline
\end{tabular}
\caption{\small{Coefficients for symmetric schemes of order 6 for the Vlasov--Poisson equation in the general case ($d > 1$) for $\sigma=9, 11$ and $13$.
}}
\label{tab.rkn.3}
\end{table}

In the one-dimensional case, $d=1$, we have in addition $[[T,U],U] = 2 m  U$, so that the operators $C_i$ in scheme (\ref{sche.1}) are simply
$C_i = (b_i + 2 c_i m \tau^2) U$. But in this case one can do even better since this feature leads to additional simplifications also
at higher orders in $\tau$. Specifically, a straightforward calculation shows that
\begin{eqnarray}  \label{modifpotb}
   W_{5,1} & = &  [U,[U,[T,[T,U]]]]   = 4 m^2 U  \nonumber  \\
    W_{7,1} & = & [U,[T,[U,[U,[T,[T,U]]]]]] = -8 m^3 U  \\
    W_{7,2} & = &  [U,[U,[U,[T,[T,[T,U]]]]]] = 0,  \nonumber
\end{eqnarray}
and the number of order conditions is further reduced, as shown in the third line of Table \ref{tab.rkn.1} ($d=1$). Although this reduction only manifests at orders
higher than six, we can incorporate the flows of $W_{5,1}$ and
$W_{7,1}$ into the composition, namely we can replace the $e^{b_i \tau U}$ in  (\ref{dhom.2}) by  $e^{\tau D_i}$, where
\begin{eqnarray*}
   D_i & = & b_i U + c_i \tau^2  [[T,U],U] + d_i \tau^4  W_{5,1} + e_i \tau^6  W_{7,1}  \\
    & = & (b_i + 2 c_i m \tau^2 + 4 d_i m^2 \tau^4 - 8 e_i m^3 \tau^6) U.
\end{eqnarray*}
In this way it is possible  to reduce the number of exponentials in the composition and obtain more efficient integrators tailored for this special
situation. In the particular case of a $6$th-order symmetric scheme it turns out that the $d_i$ and $e_i$ coefficients can be used to vanish
some of the conditions at order seven, and thus reduce the overall error. The composition
\begin{equation}   \label{sche.2}
   \psi_6^{\tau} = \e^{\tau D_1} \, \e^{ a_1 \tau T} \, \e^{\tau D_2} \, \e^{ a_2 \tau T} \, \e^{\tau D_3} \, \e^{ a_3 \tau T} \,
      \e^{\tau D_3} \, \e^{ a_2 \tau T} \, \e^{\tau D_2} \, \e^{ a_1 \tau T} \, \e^{\tau D_1}
\end{equation}
with
\begin{eqnarray*}
   D_1 & = & (b_1 + 2 c_1 m \tau^2) U \\
   D_2 & = & (b_2 + 2 c_2 m \tau^2 + 4 d_2 m^2 \tau^4) U \\
   D_3 & = & (b_3 + 2 c_3 m \tau^2 + 4 d_3 m^2 \tau^4 - 8 e_3 m^3 \tau^6) U
\end{eqnarray*}
and coefficients collected in Table \ref{tab.rkn.2} ($d=1$) turns out to be particularly efficient, as shown in \cite{bedros}. Here the parameters
$c_i$, $d_i$ and $e_i$ have been chosen to satisfy $4$ out of $8$ conditions at order $7$.

\begin{table}[ht]
\centering
\begin{tabular}{cl}
\noalign{\smallskip}\hline\noalign{\smallskip}
\begin{tabular}{c}
$\psi_6^{\tau}$ of the form \eqref{sche.2} \\ with $\sigma=11$, $ \;\; d=1$
\end{tabular}
&  \quad
 \begin{tabular}{rcl}
$a_1$ &=&  $0.168735950563437422448196$    \\
$a_2$ &=&  $0.377851589220928303880766$ \\
$a_3$ &=& $-0.093175079568731452657924$ \\
$b_1$ &=&  $0.049086460976116245491441$  \\
$b_2$ &=&  $0.264177609888976700200146$  \\
$b_3$ &=&  $0.186735929134907054308413$  \\
$c_1$ &=& $-0.000069728715055305084099$ \\
$c_2$ &=& $-0.000625704827430047189169$ \\
$c_3$ &=& $-0.002213085124045325561636$ \\
$d_2$ &=& $-2.916600457689847816445691 \cdot 10^{-6}$ \\
$d_3$ &=&  $3.048480261700038788680723 \cdot 10^{-5}$  \\
$e_3$ &=&  $4.985549387875068121593988 \cdot 10^{-7}$ \\
   \end{tabular} \\
\noalign{\smallskip}\hline
\end{tabular}
\caption{\small{Coefficients for a symmetric scheme of order 6 for the Vlasov--Poisson equation in dimension $d=1$.
}}
\label{tab.rkn.2}
\end{table}

The methods we have considered here are left-right symmetric compositions whose first flow corresponds to the functional $\Uc$. It is clear, however,
that similar compositions but now with the first flow corresponding to $\Tc$ can be considered. In that case, the schemes read
\begin{equation}   \label{modifpot2}
   \psi_p^{\tau} = \e^{ a_1 \tau T} \, \e^{\tau C_1} \,  \cdots \, \e^{\tau C_1} \, \e^{ a_1 \tau T}.
\end{equation}
This corresponds to a different class of methods, in general with a different behavior and efficiency, since this
problem possesses a very particular algebraic structure that is not preserved by  interchanging $T$ and $U$.
We have
also analyzed 6th-order schemes of this class, but we have not found better integrators than those collected in Tables \ref{tab.rkn.3}  and \ref{tab.rkn.2} in our
numerical experiments.

\section{Numerical examples}
\label{sec:num}
This section is devoted to numerical illustrations of the previous splitting methods
in the cases $d=1$ and $d=2$ (in \eqref{vlsp}).

The splitting methods introduced above enable to reduce the numerical resolution of
the Vlasov--Poisson problem \eqref{vlsp} to one-dimensional linear transport problems
of the form
\begin{equation}
\label{eq_a}
\partial_t f + a \partial_z f = 0, \qquad  f(t^n, z) = g(z), \qquad z\in \T^1,
\end{equation}
where $z$ can denote the spatial direction $x$ or the velocity direction $v$,  $a$ is a
coefficient which does not depend on the advected direction $z$, and $g$ denotes
an initial condition given on a uniform grid of $N$ points. Typically, $a$ is the component of the vector $v$ or of the electric field frozen at some grid point in the $x$-variable.

To deal with the one-dimensional advection equations,
a semi-Lagrangian method is used (see \cite{despres, ChaDeMe2011, CroMeSo2010}).
Since the characteristics can be solved exactly in this case ($a$ does not depend on $z$),
the error produced by the scheme comes from the splitting (error in time) and from
the interpolation step (error in $x$ and $v$). Note that the interpolation is performed
using high-order Lagrange polynomials (of order $17$ in practice) so that the numerical
solution of \eqref{eq_a} writes
$$
f(t^{n+1}, z_i) \approx \mathcal{I} g( z_i - a\tau),
$$
where  $\mathcal{I}$  is an interpolation operator (piecewise Lagrange interpolation in our case).
We refer the reader to  \cite{BeMe2008, despres, ChaDeMe2011, CroMeSo2010}) for more details.
After each advection in the velocity direction ($\Uc$ part),
the Poisson equation \eqref{eq:elec}  is solved to update the electric potential $\phi$.
Note that in the case $d=2$, the Hamiltonian splitting leads to $2$-dimensional advections
$\Uc$ and $\Tc$. These subproblems are split again leading to one-dimensional
advections; this does not introduce additional errors since it concerns linear advection
for which this subsplitting is exact. The numerical resolution of the Poisson equations \eqref{eq:elec}
and  \eqref{jacE} to get $\phi$ and $K$ is performed using a spectral method. Their derivatives are computed using high
order finite differences.

We consider the following initial condition for \eqref{vlsp} with $d=1$
\begin{equation}
\label{ci1}
f(t=0, x, v)=\frac{1}{\sqrt{2\pi}}\exp(-v^2/2) (1+0.5 \cos(k x)),
\end{equation}
with  $x\in [0, 2\pi/k], \;  v\in [-v_{\max}, v_{\max}]$,  $v_{\max}=8$ and $k=0.5$.
In the case $d=2$, the following initial condition for \eqref{vlsp} is chosen
\begin{equation}
\label{ci2}
 f(t=0, x, y, v_x, v_y) = \frac{1}{2\pi}\exp(-(v_x^2+v_y^2)/2)\left( 1+0.5 \cos(k x)\cos( k y)\right),
\end{equation}
where $x, y\in [0, 2\pi/k], v=(v_x, v_y)\in [-v_{\max}, v_{\max}]^2$, $v_{\max}=8$ and $k=0.5$.

We are interested in the total energy conservation
$\Hc(f)$ given by \eqref{scarlatti}. Indeed, this quantity is
theoretically preserved by \eqref{vlsp} for all times, so it represents an interesting diagnostic. For a given
time splitting, we introduce the discrete total energy $\Hc(f_h)(t)$ (integrals in phase space are replaced by summations)
where $f_h$ denotes the solution of the splitting scheme
and we look at the following quantity
\begin{equation}
\label{diag_energy}
\mbox{err}_\Hc = \max_{t\in [0, t_{\mathrm{f}}]} \left|\frac{\Hc(f_h)(t)}{\Hc(f)(0)}-1\right|,
\end{equation}
where $t_{\mathrm{f}}>0$ is the final time of the simulation. We are also interested in the $L^2$ norm $\|f_h(t)\|_{L^2}$
of $f_h$ (which is also preserved with time) and we plot the quantity
\begin{equation}
\label{diag_Lp}
\mbox{err}_{L^2} = \max_{t\in [0, t_{\mathrm{f}}]} \left|\frac{\|f_h(t)\|_{L^2}}{\|f(0)\|_{L^2}}-1\right|,
\end{equation}

Different splitting will be studied regarding these quantities to compare their relative performances.
First, we consider some splitting methods from the literature: the well-known $2$nd-order Strang splitting (STRANG,
$\sigma=3$ flows per step size, even if we take $\sigma=2$
in all the figures, since the last flow can be concatenated with the first flow at the next iteration), the so-called triple jump $4$th-order composition
\cite{Yoshida1990} (3JUMP, $\sigma=7$ flows) and the $6$-th order splitting method proposed in \cite{BlaMo2000} (06-23, $\sigma=23$ flows).
Then, the splitting methods introduced in this work are considered. When $d=1$, the method of Table \ref{tab.rkn.2} (06-11, $\sigma = 11$ flows), and in the
case $d>1$ the schemes of Table \ref{tab.rkn.3}: 06-9, 06-11 and 06-13, with $\sigma=9, 11$ and $13$ flows, respectively.

In the following figures, we choose a final time $t_{\mathrm{f}}$ and the quantities \eqref{diag_energy}
and \eqref{diag_Lp} are plotted
as a function of  $\sigma/\tau$, where $\sigma$ is the number of flows
of the considered method and $\tau$ is the
time step used for the simulation.
This choice ensures that all the diagnostics are obtained with
a similar CPU cost. In the sequel, we consider $70$ different time steps
in $[0.125, 8]$ for $d=1$ and $100$ different time steps belonging to the interval $[0.1, 30]$ for $d=2$.
Finally, we denote by $N$ the number of points per direction we use to sample
uniformly the phase space grid.

\noindent In Figure \ref{fig:nrj_vp1d}, we first focus on the $d=1$ case. We plot
the quantity relative to the total energy err$_{\Hc}$ defined in \eqref{diag_energy}
for STRANG, 3JUMP, 06-23 and our 06-11 (see Table \ref{tab.rkn.2}) using $N=256$ points
per direction and $t_{\mathrm{f}}=16$.
The expected orders of the different methods are recovered. However, even if 06-23 and 06-11 are
both of $6$th-order, 06-11 presents a better behavior since the total energy is better
preserved up to two orders of magnitude than 06-23, with a comparable time CPU.
Note that the 06-11 scheme has also been used with success
in the one-dimensional context in \cite{bedros}.
In Figures \ref{fig:nrj_vp1d_N}, the same diagnostics as before is shown, but with two smaller values of $N$.
For $N=64$, we can also observe the plateau for small $\tau$ which reveals the phase space error.
The level of this plateau can be decreased by increasing $N$.

On Figure \ref{fig:nrj_vp1d_time}, the time evolution of $\Hc(f_h)(t)$ and $\| f_h(t)\|_{L^2}$ are displayed
for the four splitting methods with different $N$ at a almost constant CPU time: STRANG with $\tau=1/8$, 3JUMP with $\tau=0.4$,
06-23 with $\tau=4/3$ and 06-11 with $\tau=0.64$. It appears that the conservation of the total energy is very well preserved
for 06-11. For the conservation of the $L^2$ norm, the benefit of high-order splitting is not so clear since
all the curves are nearly superimposed. When $N$ increases, we observe that the eruption time increases;
the eruption time corresponds to the time at which the finest scale length of $f$ reaches the phase space grid size.
After this time, the error rapidly blows up (see \cite{WatanSug2004}).
Finally, on Figure \ref{fdis}, we display the whole phase space distribution function at time $t_{\mathrm{f}}=16$
obtained with 06-11. We can observe the fine structures (filaments) which are typically developed in this nonlinear Landau test case.

\noindent Next, we focus on the $d=2$ case.
In Figure \ref{fig:nrj_vp2d_standard}
we plot the quantity relative to the total energy err$_{\Hc}$ defined in \eqref{diag_energy} as a function of
$\sigma/\tau$ for $N=64$ and $t_{\mathrm{f}}=60$.
First, on the left part of Figure \ref{fig:nrj_vp2d_standard}, methods of the literature are displayed: STRANG (order 2), 3JUMP (order $4$), and
06-23 (order $6$).
This diagnostic enables us to recover the expected order of the different methods.
Then, on the right part of Figure \ref{fig:nrj_vp2d_standard} we focus on our new methods 06-9, 06-11 and 06-13 (see Table  \ref{tab.rkn.3}).
All the methods in this figure are of order $6$, so that
it enables to study the influence of the number of flows $\sigma$ (the reference
06-23 method is also displayed) on the total energy conservation.
Even if all the methods are of order $6$,
they have not the same precision. Indeed, adding some flows in the splitting method enables in our context
to generate a more efficient method. Two explanations can be made: first, the coefficients can be
chosen smaller and few of them are negative and second, the error term can be optimized
(see Section \ref{derivation}).
Finally, the 06-13 method appears to be
the best method, reaching an error of about $10^{-10}$ with $\tau\approx 0.2$.

\noindent In Figure \ref{fig:l2norm_standard}
we plot the quantity err$_{L^2}$ defined in \eqref{diag_Lp} as a function of $\sigma/\tau$
for the different methods in the case $d=2$.
This diagnostic enables to quantify the dissipation (small details of the solution are eliminated) of the numerical methods (see \cite{CroMeSo2010}).
As mentioned in the case $d=1$, the benefit of high-order time integrator is not very clear.
We can also remark that the influence of the number of flows is not very significative
(see right part of Figure  \ref{fig:l2norm_standard}); indeed, when $\tau$ is small
we observe that all the methods converges towards a constant (which is $0.04$ for 3JUMP, $0.03$ for 06-23
and a close value of $0.03$ for the other methods). Here, 3JUMP and 06-9 are the worst methods; this may be linked
with the fact that they contain large negative coefficients and then they present important zigzag (see \cite{hairer}).

	\begin{figure}
		\includegraphics[width=12cm]{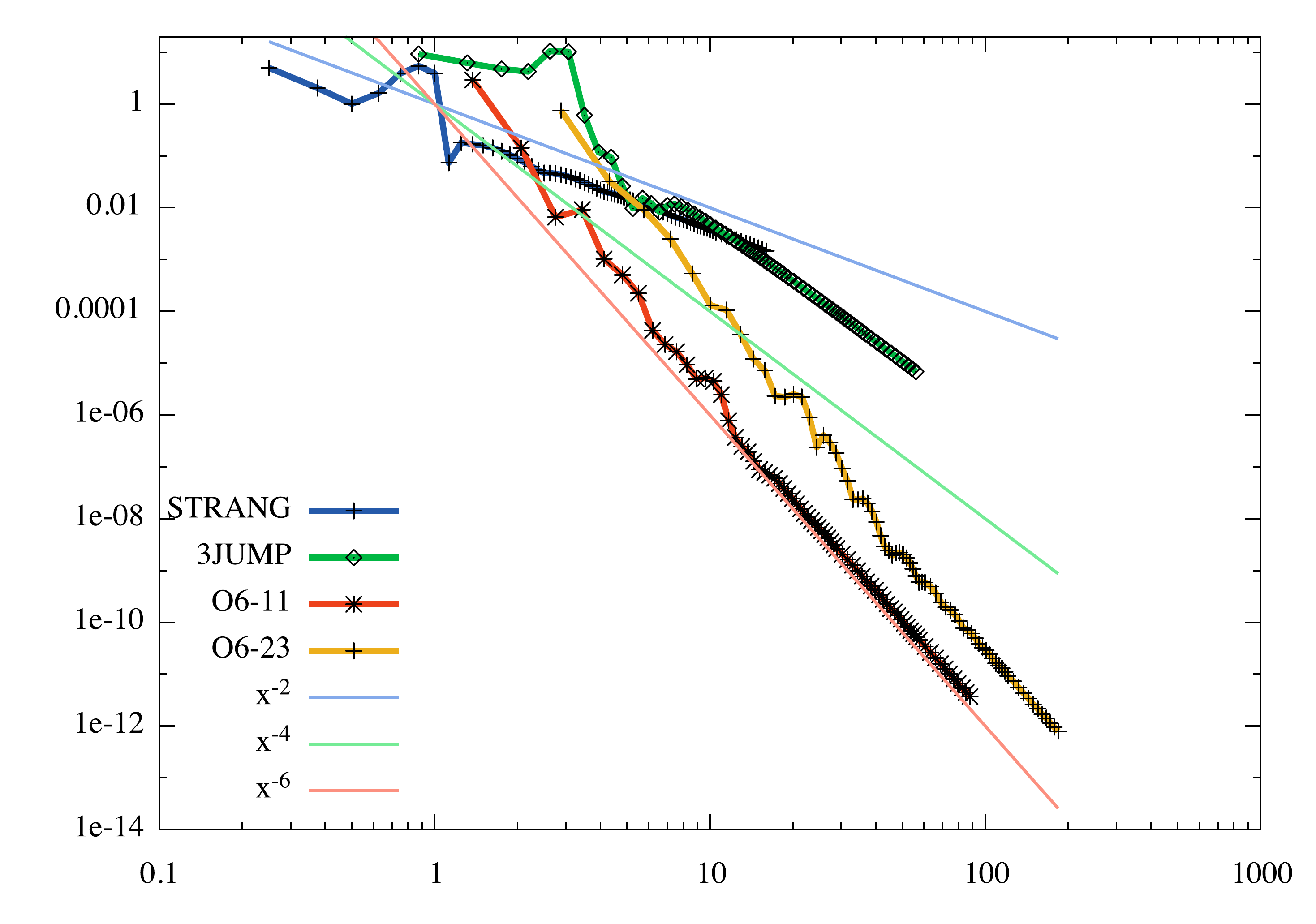}
		\caption{Case $d=1$. err$_\Hc$ (defined by \eqref{diag_energy}) as a function of $\sigma/\tau$ where $\sigma$ is the number of flows
		and $\tau$ the time step, for the different splitting methods. $N=256$. }
		\label{fig:nrj_vp1d}
	\end{figure}

\begin{figure}
	\begin{tabular}{cclll}
		\includegraphics[width=6cm]{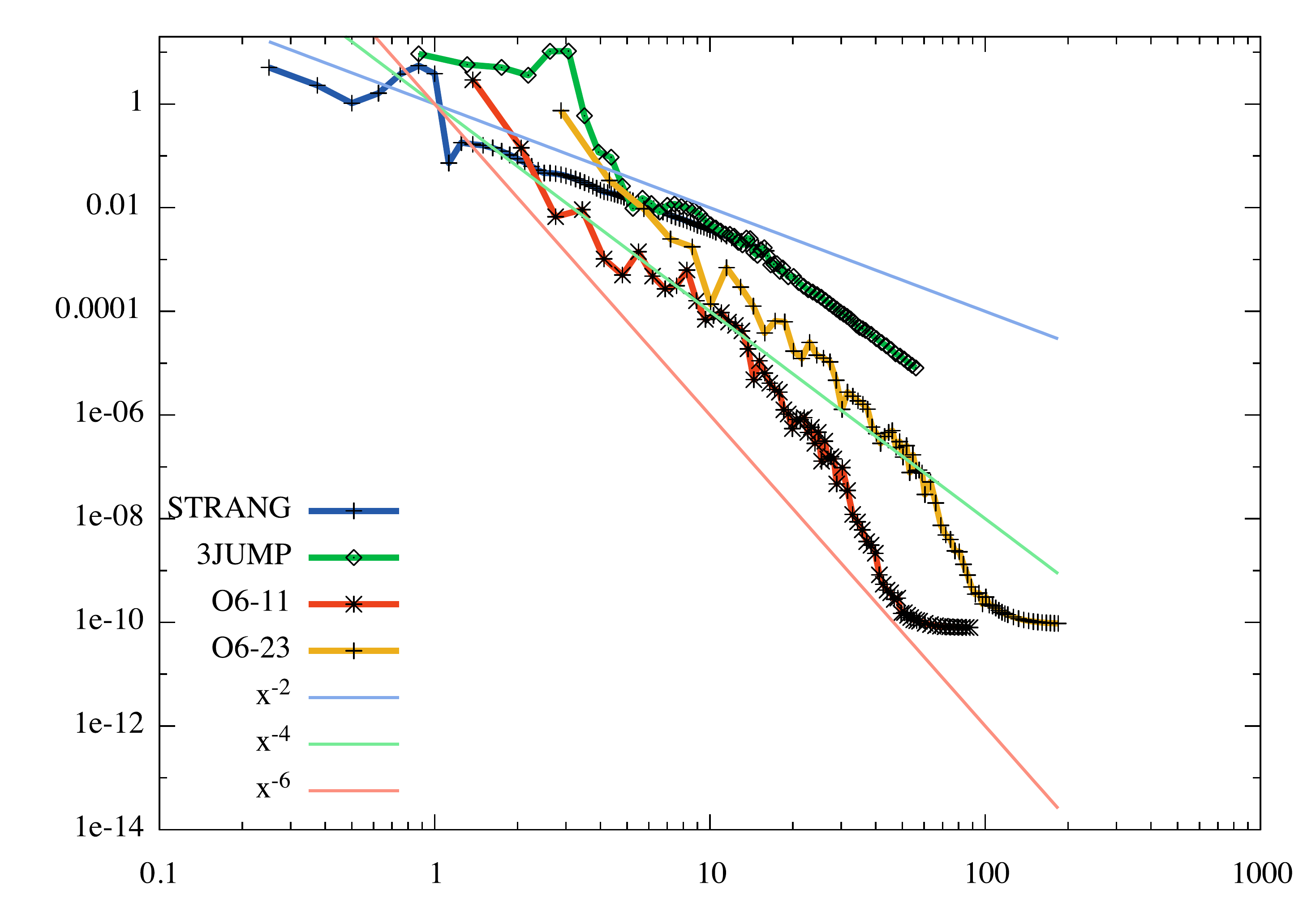}&
	\includegraphics[width=6cm]{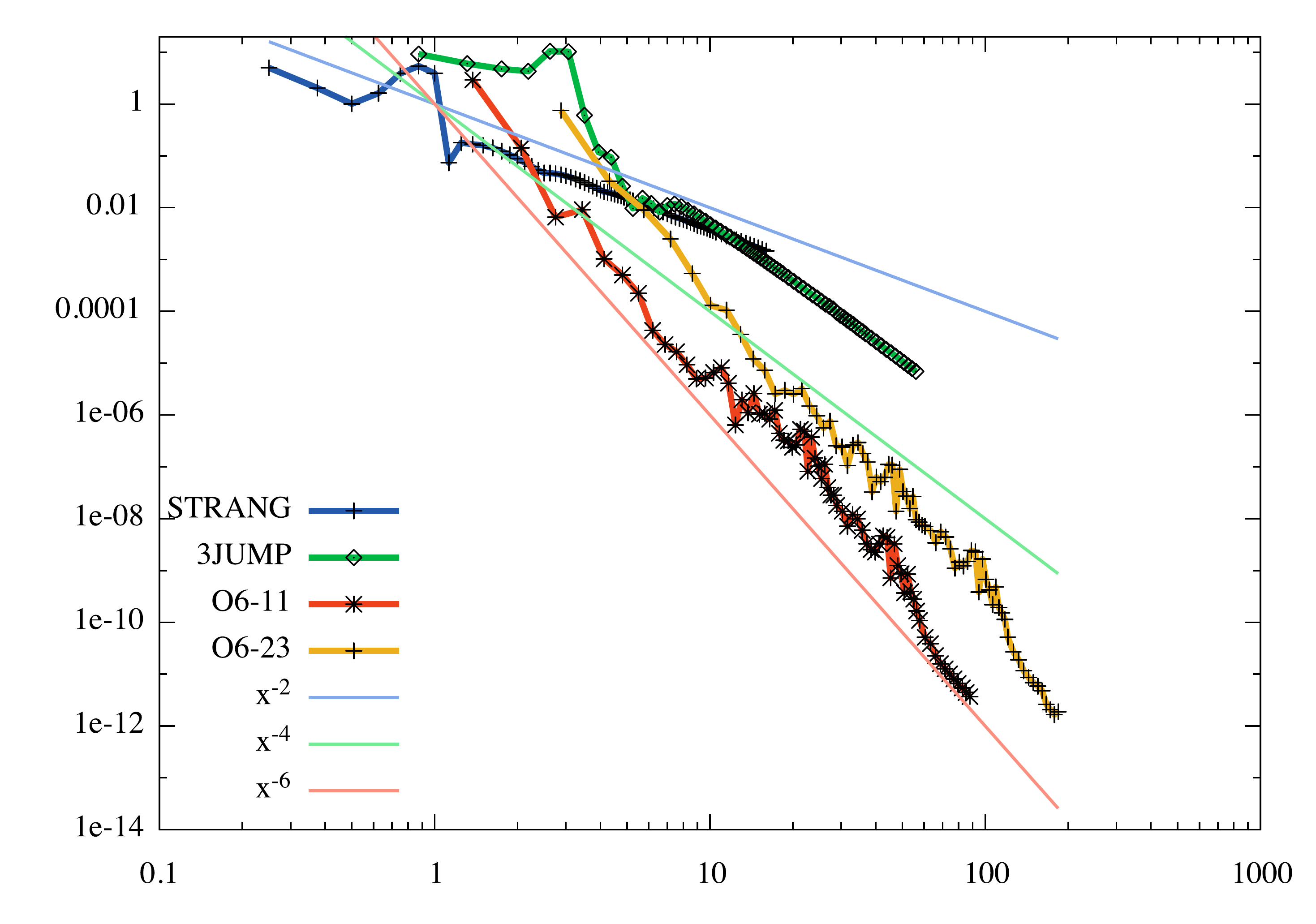}&
	\end{tabular}
		\caption{Case $d=1$. err$_\Hc$ (defined by \eqref{diag_energy}) as a function of $\sigma/\tau$ where $\sigma$ is the number of flows and $\tau$ the time step, for the different splitting methods. From left to right, $N=64, 128$.  }
\label{fig:nrj_vp1d_N}
\end{figure}

\begin{figure}
	\begin{tabular}{cclll}
		\includegraphics[width=6cm]{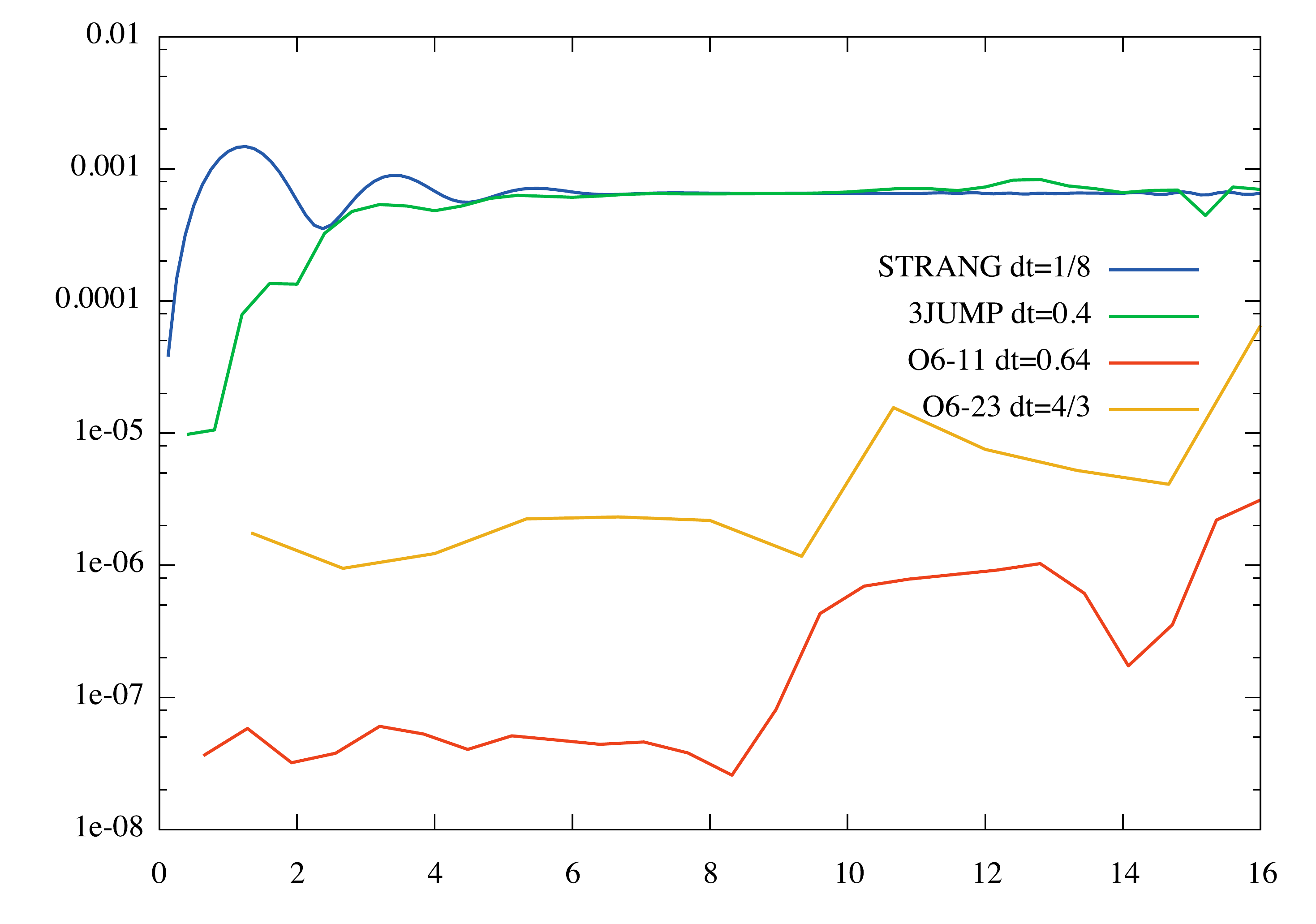}&
                 \includegraphics[width=6cm]{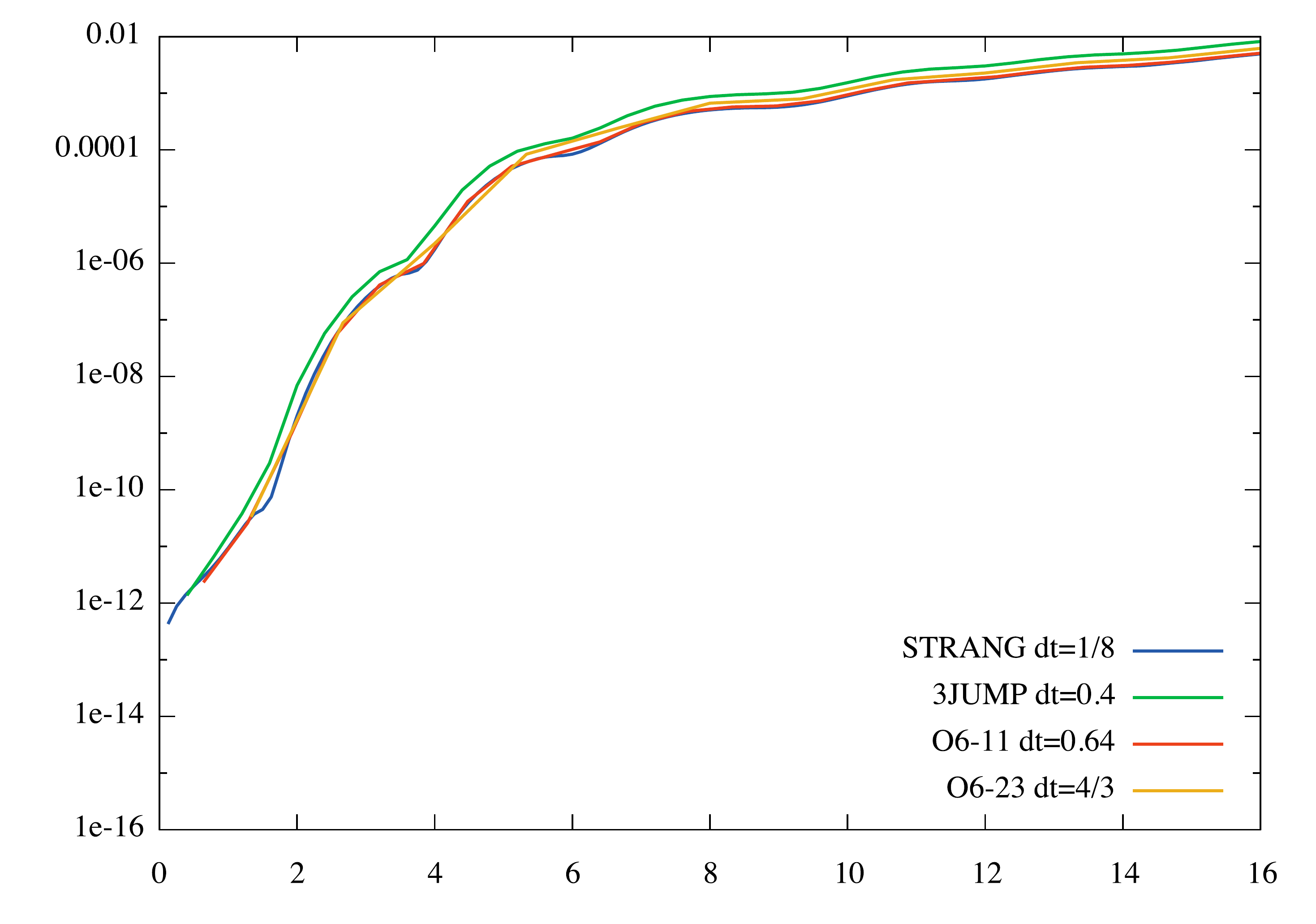}\\
	\includegraphics[width=6cm]{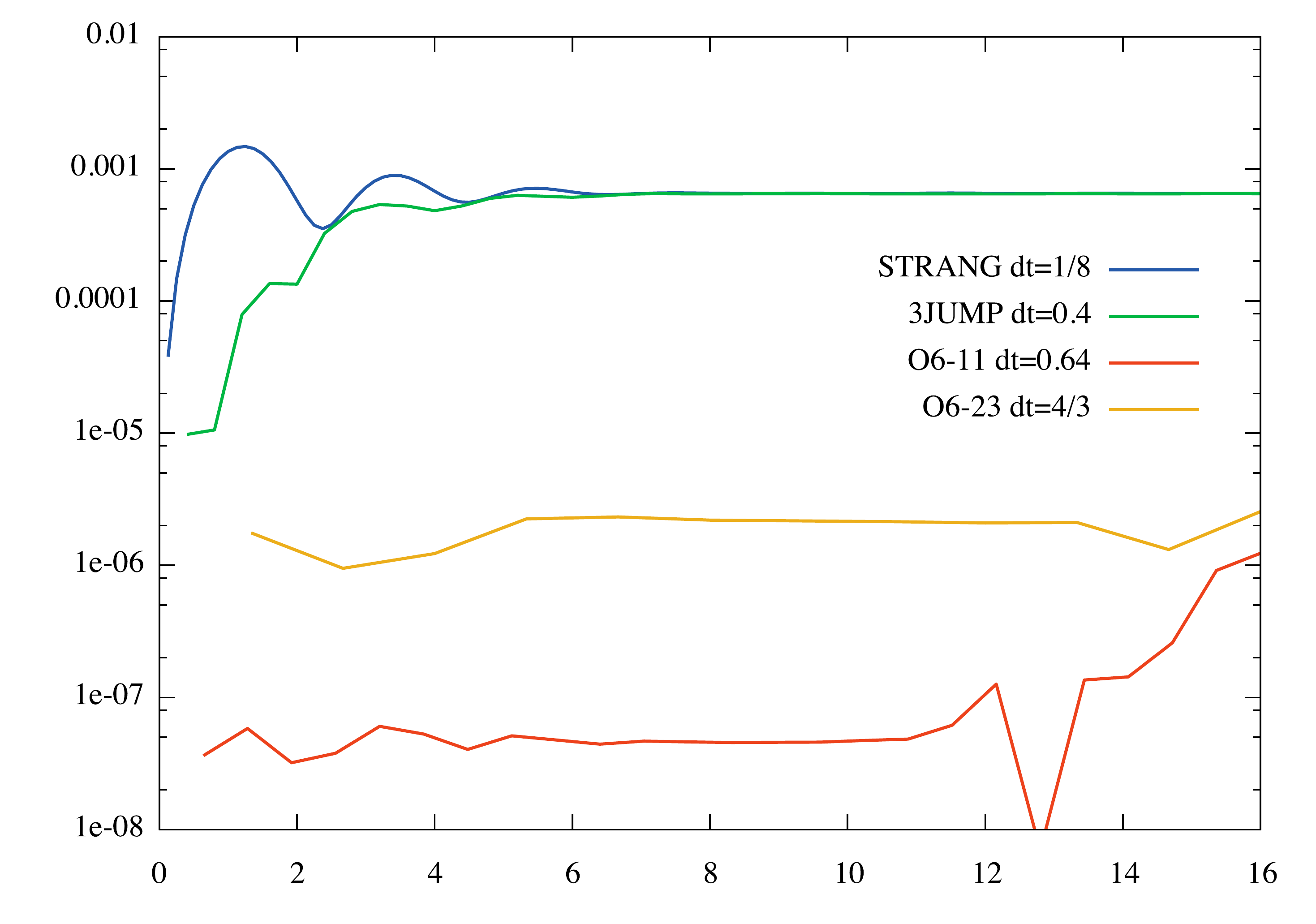}&
		\includegraphics[width=6cm]{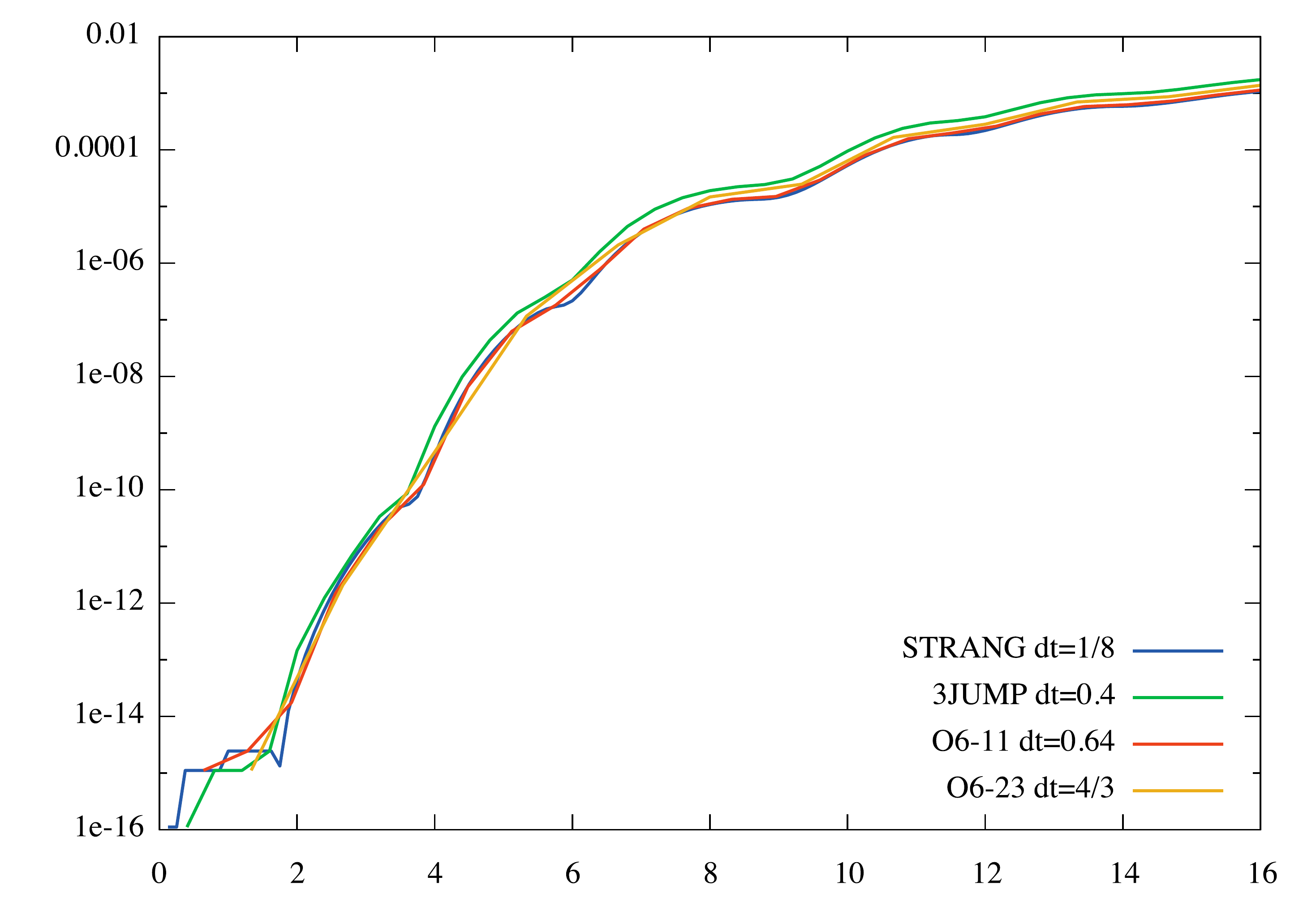}\\
	\includegraphics[width=6cm]{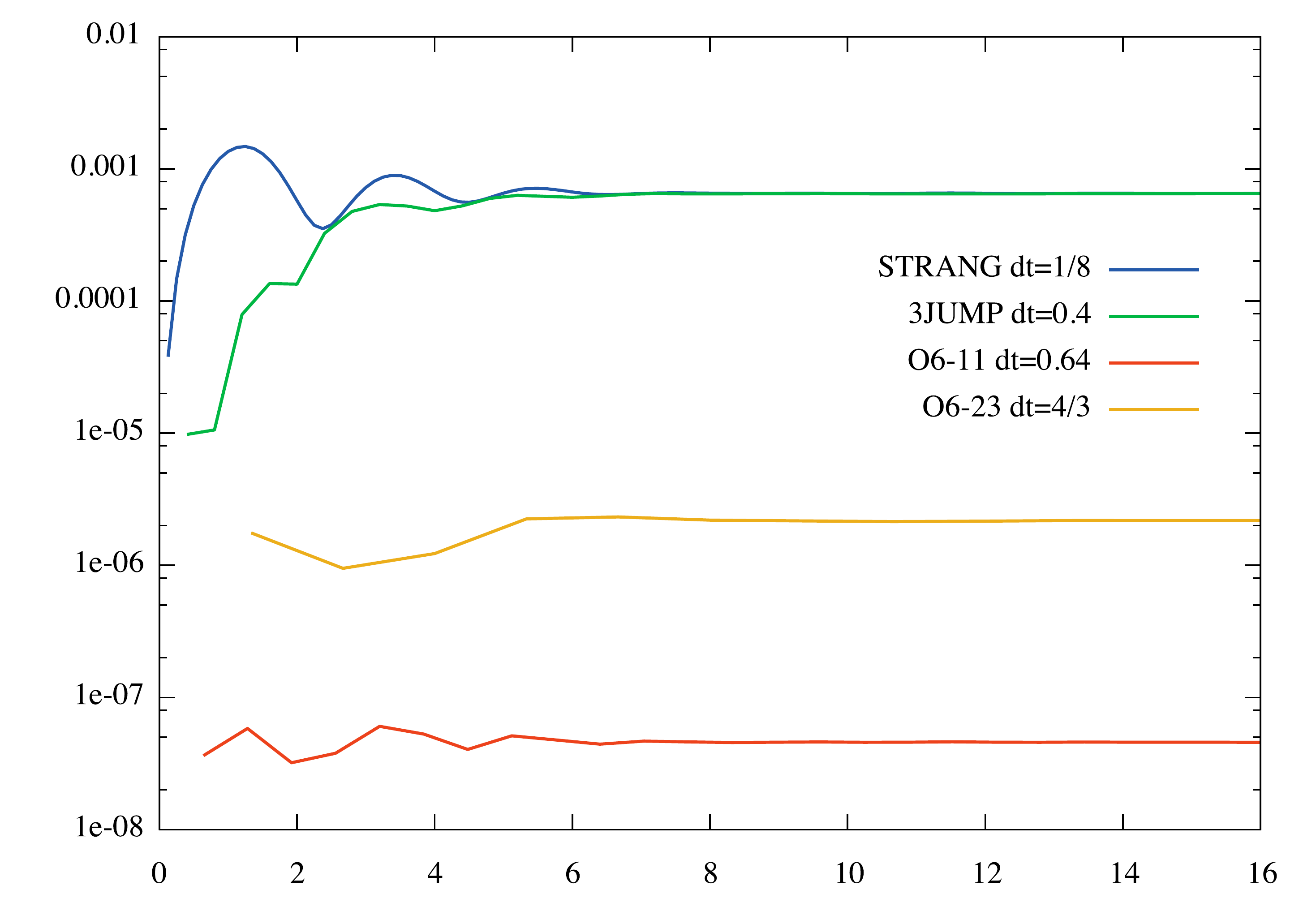}&
		\includegraphics[width=6cm]{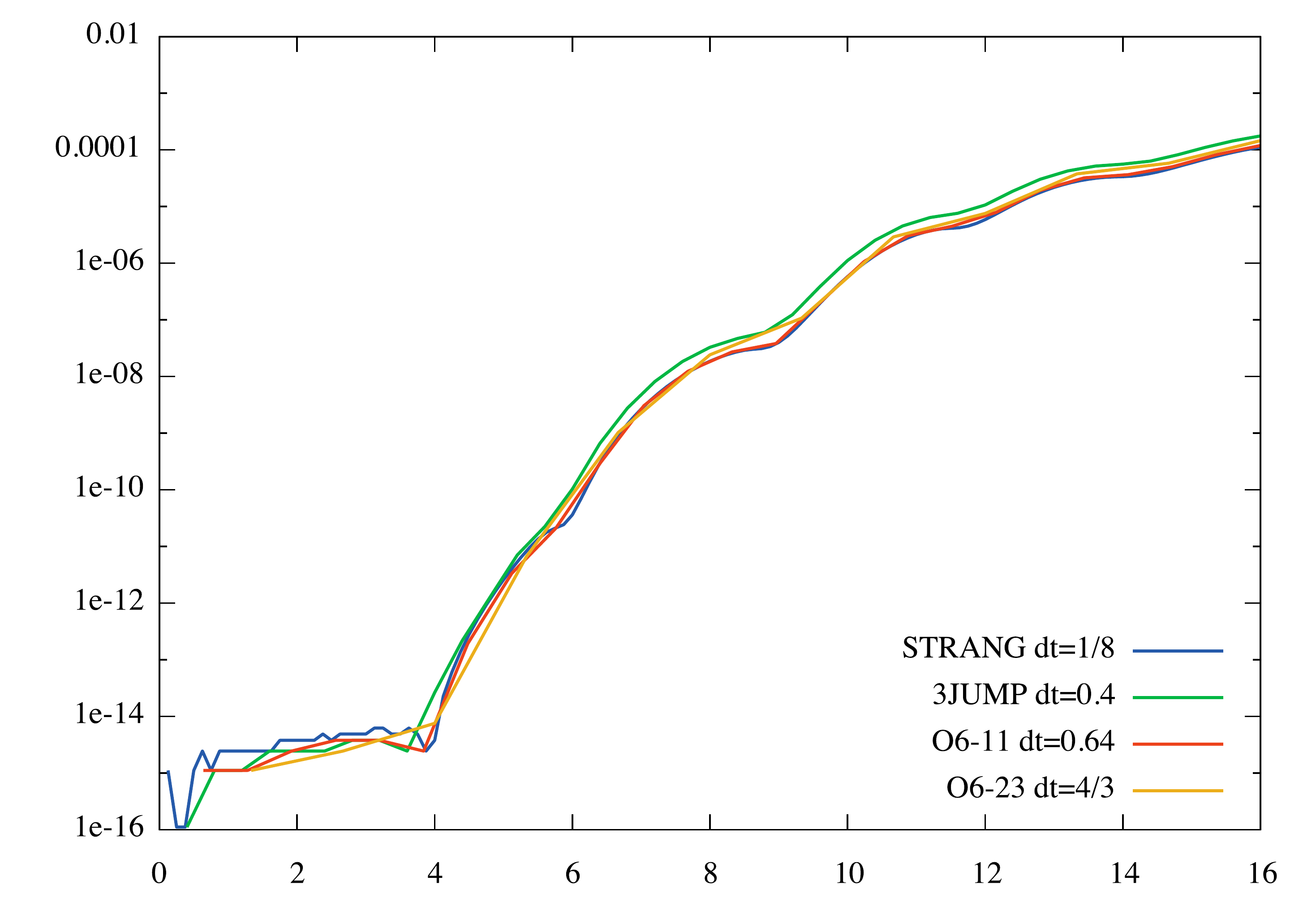}
	\end{tabular}
		\caption{Case $d=1$. Time history of $\Hc(f_h)$ (left column) and of $\|f_h(t)\|_{L^2}$ norm (right column).
		From top to bottom, $N=64, 128, 256$. Comparison of STRANG, 3JUMP, 06-11 and 06-23 at almost constant 		CPU time. }
	\label{fig:nrj_vp1d_time}
\end{figure}

	\begin{figure}
		\includegraphics[width=6cm]{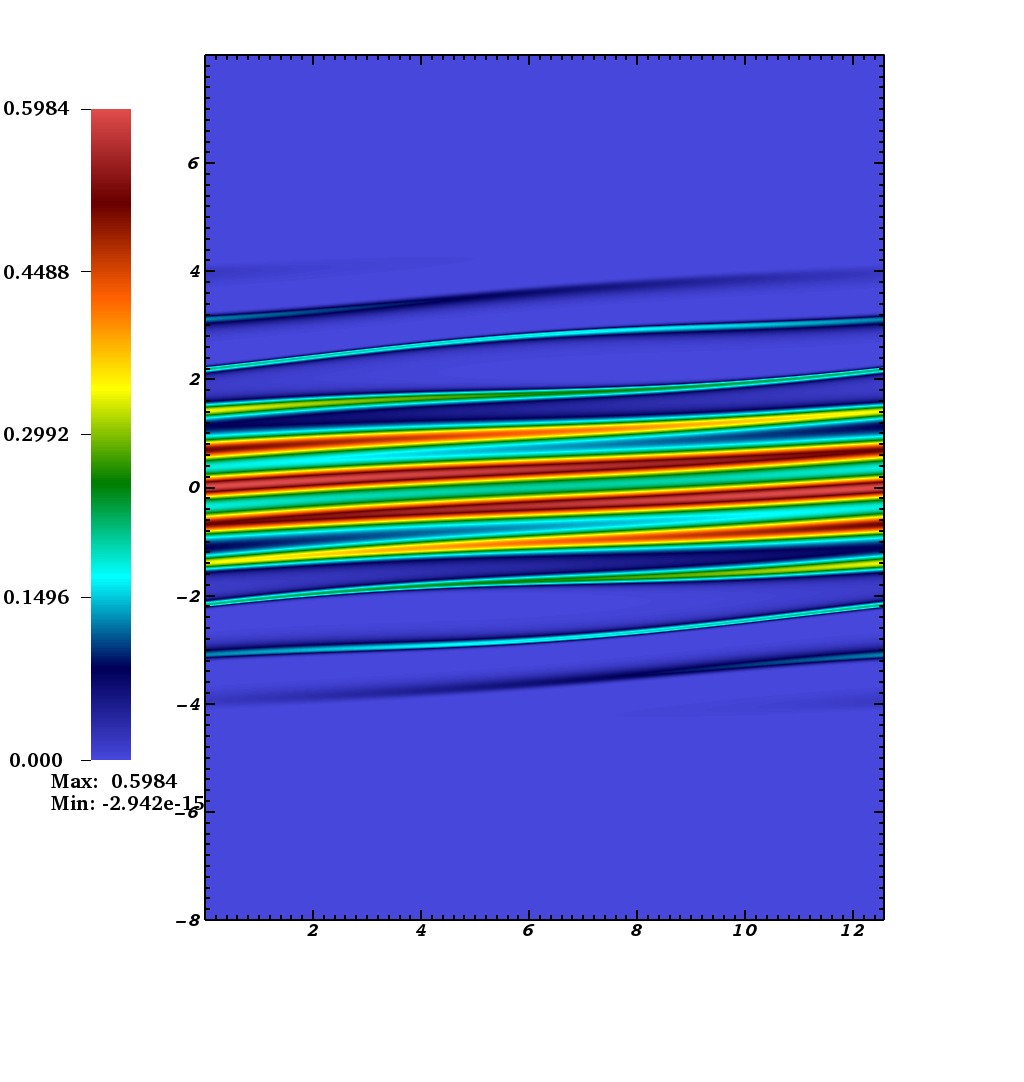}
		\caption{Case $d=1$. Phase space distribution function $f(t=16, x, v)$ obtained with 06-11. $N=1024$. }
		\label{fdis}
	\end{figure}

	\begin{figure}
	\begin{tabular}{cclll}
\includegraphics[width=6cm]{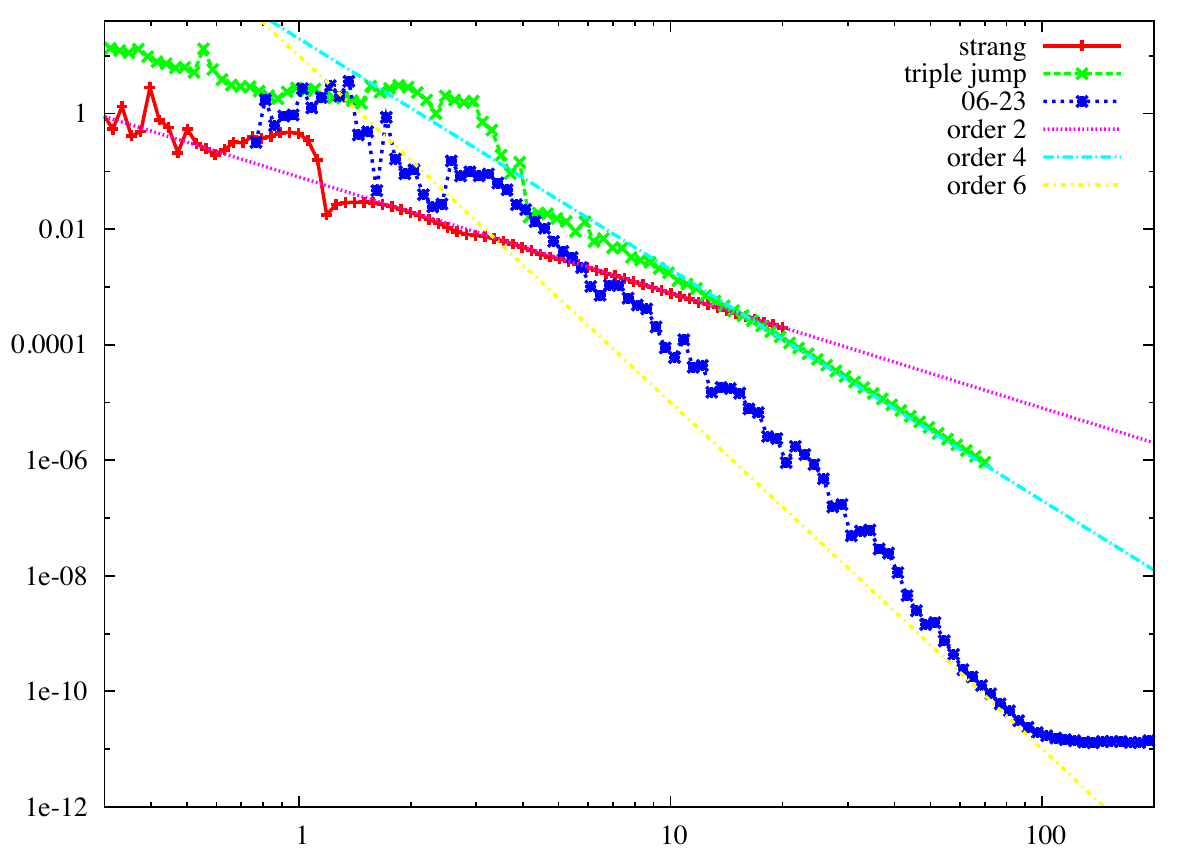}	&
\includegraphics[width=6cm]{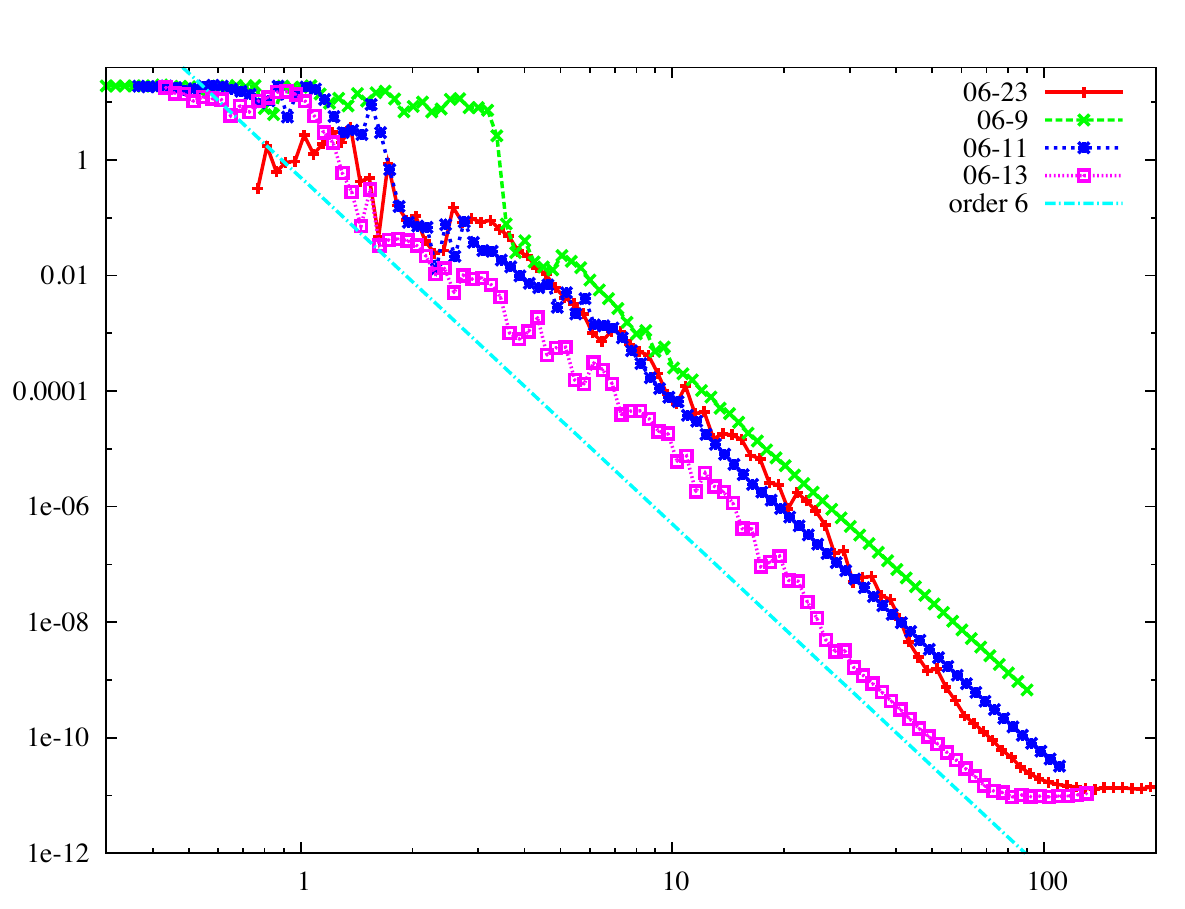}
\end{tabular}
		\caption{Case $d=2$: err$_\Hc$ (defined by \eqref{diag_energy}) as a function of $\sigma/\tau$ where $\sigma$ is the number of flows
		and $\tau$ the time step, for the different splitting methods. $N=64$.}
		\label{fig:nrj_vp2d_standard}
	\end{figure}

	\begin{figure}
	\begin{tabular}{cclll}
		\includegraphics[width=6cm]{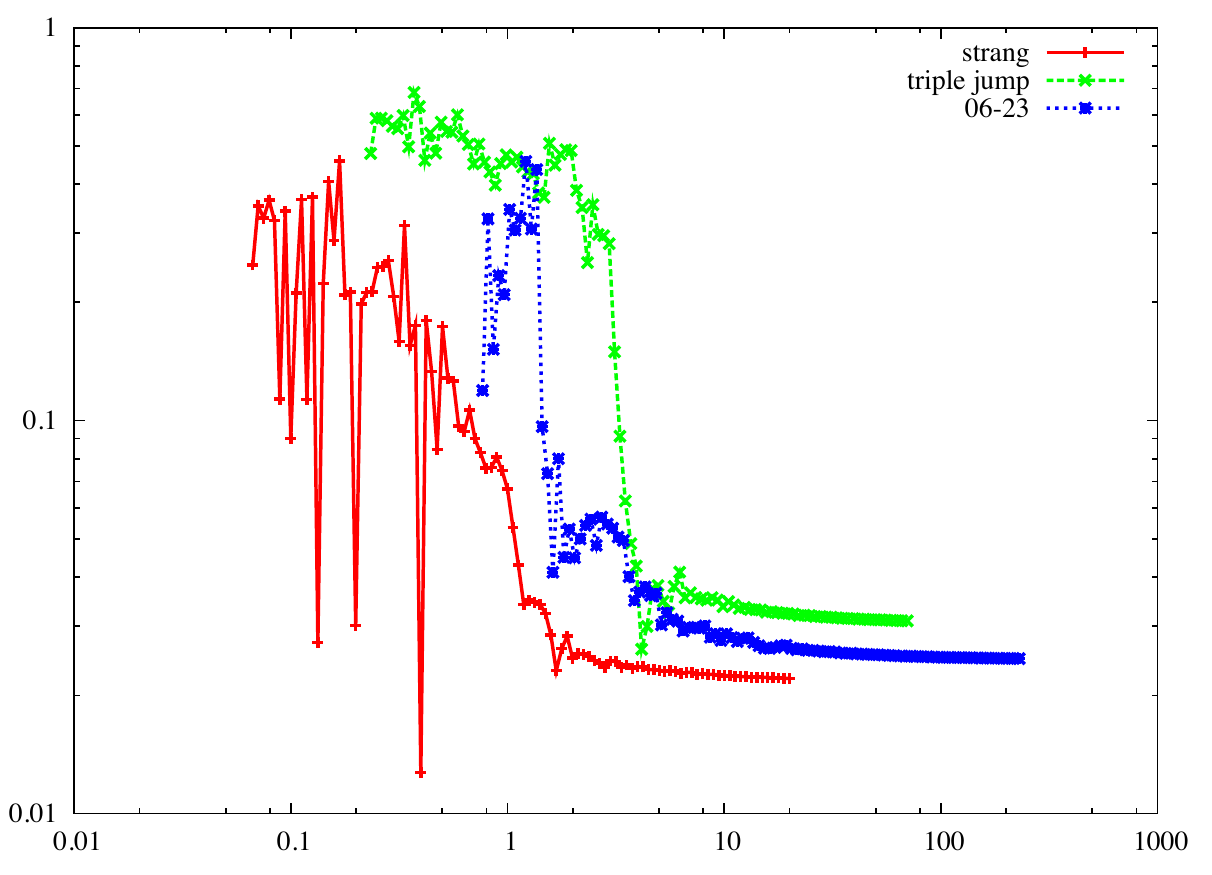} &
		\includegraphics[width=6cm]{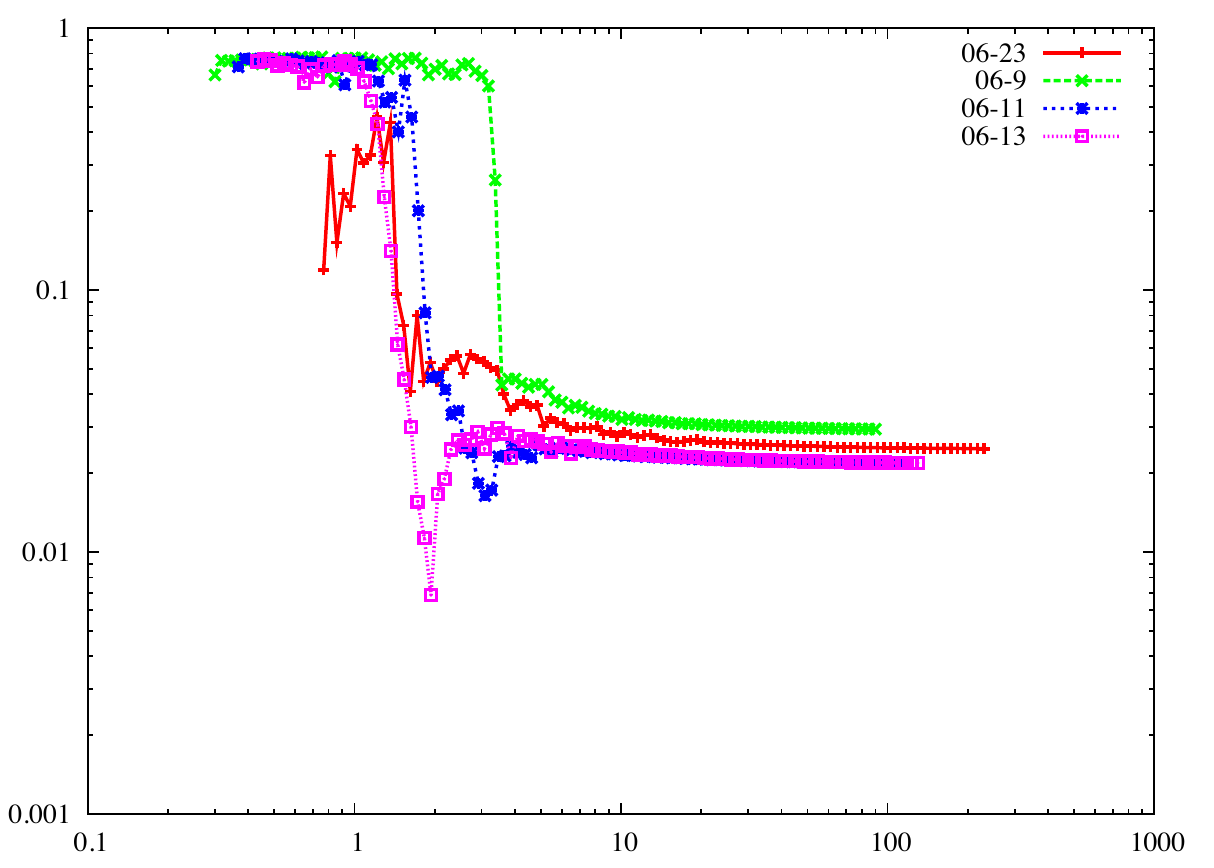}
		\end{tabular}
		\caption{Case $d=2$: err$_{L^2}$ (defined by \eqref{diag_Lp}) as a function of $\sigma/\tau$ where $\sigma$ is the number of flows
		and $\tau$ the time step, for the different splitting methods. $N=64$.}
		\label{fig:l2norm_standard}
	\end{figure}


\section{Convergence estimates}
\label{sec:conv}

This last section is devoted to the rigorous mathematical analysis of Vlasov--Poisson equations and their approximation by splitting methods of the form \eqref{dhom.1} satisfying the order conditions to ensure \eqref{dhom.1b}.

For a given multi-index $p  = (p_1,\ldots,p_d) \in \N^d$, we denote by $\partial_x^p$ the multi-derivative $\partial_{x_1}^{p_1} \cdots \partial_{x_d}^{p_d}$. Moreover, we set $|p| = p_1 + \ldots + p_d$. Similarly, we set $v^m := v_1^{m_1} \cdots v_d^{m^d}$ for $v = (v_1,\dots, v_d) \in \R^d$ and $m = (m_1,\ldots,m_d) \in \N^d$.
As functional framework, we will consider the spaces $\Hc^r_\nu$ equipped with the norms
\begin{equation}
\label{normeHs}
\Norm{f}{\Hc^r_\nu}^2 = \sum_{\substack{(m,p,q) \in \N^d \times \N^d \times \N^d \\ |p| + |q| \leq r \\Ê|m| \leq \nu}} \int_{\R^d} \int_{\T^d} | v^m \partial_x^p \partial_v^q f(x,v)|^2 \dd x \dd v,
\end{equation}
where, $\partial_x^p$ and $\partial_v^q$ denote the usual multi-derivative in the $x$ and $v$ variables.
In such spaces - already  considered in \cite{Degond86} -  and for $r$ and $\nu$ large enough,  the Vlasov--Poisson equation is well-posed and satisfies stability estimates ensuring the convergence of stable and consistent numerical methods, see Theorem \ref{th63appendix} and Lemma \ref{lem:1} below for precise estimates.

Before giving a complete proof of these results, we will state some useful estimates. In the following, we will denote by $L^q_x$ and $L^q$ for $q = 2$ and $q = \infty$ the standard $L^q$ spaces on $\T^d$ and $\T^d \times \R^d$ respectively. Similarly, for $r \geq 0$, $H^r_x$ and $H^r$ denote the standard Sobolev spaces on $\T^d$ and $\T^d \times \R^d$ respectively.
\begin{lemma}
\label{lem.51}
Let $\nu>d/2$. Then we have for $p = (p_1,\ldots,p_d) \in \N^d$
\begin{equation}
\label{alouette}
\Norm{\partial_x^p \phi(f)(x)}{L^2_x} \leq C \Norm{f}{\Hc^{(|p| -2)_+}_\nu},
\end{equation}
and
\begin{equation}
\label{bach}
\Norm{\partial_x^p \phi(f)(x)}{L^\infty_x} \leq C \Norm{f}{\Hc^{(|p| + \nu - 2)_+}_\nu}.
\end{equation}
\end{lemma}

\begin{proof}
For a given function $g(x,v)$, we have
$$
\left|\int_{\R^d}g(x,v) \dd v\right| \leq  \Big(  \int_{\R^d}\frac{1}{(1 + |v|^{2})^{\nu}} \dd v \Big)^{1/2} \Big( \int_{\R^d} (1 + |v|^{2})^{\nu} |g(x,v)|^2 \dd v \Big)^{1/2},
$$
as soon as $\nu > d/2$. Applying this formula, we first see with the definition \eqref{normeHs} that
$$
\int_{\T^d} \left|\int_{\R^d}g(x,v) \dd v\right|^2 \dd x \leq C \Norm{g}{\Hc^0_\nu}^2.
$$
We then obtain \eqref{alouette} by applying this formula to
$$
g(x,v) =\Delta_x^{-1} \partial_x^p f(x,v).
$$
After using the Gagliardo-Nirenberg inequality (easily obtained in Fourier)
$$
\Norm{\partial_x^p \phi(f)}{L_x^\infty} \leq C \Norm{\phi(x)}{H^{|p| + \nu}_x},
$$
for $\nu > d/2$,
we then deduce the second equation \eqref{bach}  from  \eqref{alouette}.
\end{proof}

We will first give a meaning to the expansion \eqref{eq:exp1}. To do this we will use the two following inequalities: for $r \geq 1$ and $\nu > d/2$,
\begin{equation}
\label{eq:exp2}
\Norm{v \cdot \partial_x f}{\Hc^r_\nu} \leq \Norm{f}{\Hc^{r+1}_{\nu +1}},\quad \mbox{and} \quad
\Norm{\partial_x \phi(f) \cdot \partial_v f}{\Hc_\nu^r} \leq C \Norm{f}{\Hc^{r+1}_{\nu}}^{2}.
\end{equation}
Indeed, the first is clear from the definition of the $\Hc^r_\nu$ norm. To prove the second, we use the fact that for given $(m,p,q) \in (\N^d)^3$ satisfying $|p| + |q| \leq r$ and Ê$|m| \leq \nu$ as in the definition of \eqref{normeHs}, we have
$$
\Norm{v^m \partial_x^p \partial_v^q (\partial_x \phi(f) \cdot \partial_v f)}{L^2}
\leq C \Norm{\phi(f)}{W_x^{r+1,\infty}} \Norm{f}{\Hc^{r+1}_{\nu}},
$$
where $W_x^{r+1,\infty}$ denotes the standard Sobolev space in the $x$ variable, controlling $(r+1)$ derivatives in $L^\infty$. Now, we conclude by using
$$
\Norm{\phi(f)}{W_x^{r+1,\infty}}  \leq C \Norm{f}{\Hc^{r -1 +\gamma}_\gamma},
$$
for all $\gamma > d/2$ (see \eqref{bach} above) with $d = 1,2,3$ and $\gamma < \min(\nu,2)$.
Using the Hamiltonian formalism $\mathrm{ad}_{\Tc}f = v\cdot \partial_xf$ and $\mathrm{ad}_{\Uc}f = \partial_x \phi(f) \cdot \partial_v f$, the estimates \eqref{eq:exp2} can be written
\begin{equation}
\label{prokofiev}
\Norm{\mathrm{ad}_{\Tc}f}{\Hc^r_\nu} \leq \Norm{f}{\Hc^{r+1}_{\nu+1}} \quad \mbox{and}\quad
\Norm{\mathrm{ad}_{\Uc}f}{\Hc^r_\nu} \leq C\Norm{f}{\Hc^{r+1}_\nu}^2,
\end{equation}
and a similar inequality for the operator $\mathrm{ad}_{\Hc}$. Hence, the expansion \eqref{eq:exp1} can be easily interpreted as follows (using a Taylor expansion in time):
\begin{lemma}
\label{fronsac}
Let $r > d/2 + 1$, $ d = 1,2,3$ and $N \in \N$, and $\mathcal{B}$ a bounded set of $\Hc^{r + N + 1}_{\nu + N+1}$. Then there exists $t_0$ and $C$ such that for all $t < t_0$ and all $f \in \Bc$,
$$
\Norm{\varphi^t_{\Hc}(f) - \sum_{k =0}^N \frac{t^k}{k!} \mathrm{ad}_{\Hc}^{k}f}{\Hc^r_\nu} \leq C t^{N+1}.
$$
\end{lemma}
Of course the same lemma holds for the exact flows $\varphi_{\Tc}^t$ and $\varphi_{\Uc}^t$.
Using these expansions and the identification between splitting method \eqref{dhom.1} and methods based on composition of exponentials \eqref{dhom.2} which is done using the relation \eqref{eq:exp1}, we can make precise the notion of order that we consider in this paper. The algebraic conditions analyzed in Section \ref{derivation} and the previous estimates yields order $p$ methods in the following sense:
\begin{definition}
\label{gortoz}
The method \eqref{dhom.1} is said to be of order $p$, if there exist
$\nu_0, \nu_p \geq 0$ and $r_0,r_p\geq 0$ such that for all $\nu > \nu_0$,  $r > r_0$ and all bounded set $\mathcal{B}$ of $\Hc^{r+ r_p}_{\nu + \nu_p}$,  there exist $\tau_0$ and $C > 0$ such that for all $f \in \mathcal{B}$ and all $\tau \leq \tau_0$,
\begin{equation}
\label{order}
\Norm{\varphi_{\Hc}^\tau(f) - \psi_p^\tau(f)}{\Hc^r_\nu} \leq C \tau^{p+1}.
\end{equation}
\end{definition}
We will show in the next sections that the condition \eqref{order} implies the scheme is of order $p$ in the sense that it approximates the solution in $\Hc_\nu^r$ over a finite time interval $[0,T]$ with a precision $\mathcal{O}(\tau^p)$, provided the initial data is in $\Hc_{\nu + \nu_p}^{r + r_p}$.

\subsection{Existence of solutions}

The goal of this subsection it to prove the following result:
\begin{theorem}
\label{th63appendix}
Let $\nu > d/2$,  $r \geq  3 \nu$. There exists a constant $C_{r,\nu}$ and $L_{r,\nu}$ such that for all given $B > 0$ given and  $f_0 \in \Hc^{r+2 \nu + 1}_\nu$ such that $\Norm{f_0}{\Hc_\nu^{r + 2 \nu + 1}}\leq B$, then for all $\alpha, \beta \in [0,1]$, there exists a solution $f(t,x,v)$ of the Vlasov--Poisson equation
$$
\partial_t f + \alpha v \cdot \partial_x f - \beta \partial_x \phi(f) \cdot \partial_v f = 0,
$$
with initial value $f(0,x,v) = f_0(x,v)$, on the interval
\begin{equation}
\label{eq:T}
T := \frac{C_{r,\nu}}{1 + B},
\end{equation}
such that for all $t \in [0,T]$, we have the estimate
\begin{equation}
\label{karlsruhe}
\forall\, t \in [0,T], \quad  \Norm{f(t)}{\Hc^{r+2\nu + 1}_\nu} \leq \min(2B, e^{L_{r,\nu}  (1 + B) t}) \Norm{f_0}{\Hc^{r+2\nu + 1}_\nu}.
\end{equation}
Moreover, for two initial conditions $f_0$ and $g_0$ satisfying the previous hypothesis, we have
\begin{equation}
\label{eqlip}
\forall\, t \in [0,T], \quad \Norm{f(t) - g(t) }{\Hc^r_\nu} \leq e^{L_{r,\nu} (1+B) t}\Norm{f_0 - g_0 }{\Hc^r_\nu}.
\end{equation}
\end{theorem}

Equations \eqref{karlsruhe} and \eqref{eqlip} show that the flow is locally bounded in $\Hc^{r+2\nu + 1}_\nu$ and locally Lipschitz in $\Hc^r_\nu$ for $r$ large enough ($\nu$ being essentially $d/2$).
Before proving the theorem, we will show a stability lemma that will be useful both for the local existence of solutions and the analysis of splitting methods.

\begin{lemma}
\label{lem:1}
Let $\alpha, \beta \in \R_+$, $\nu > d/2$ and $r > 3 \nu $ be given. There exists a constant $C > 0$ such that the following holds:
Assume that  $g(t) \in \Hc^{r}_\nu$  and $f(t)$ in $\Hc^r_\nu$ are continuous function of the time,
and let $h(t)$ be a solution of the equation
\begin{equation}
\label{tuyauxpvc}
\partial_t h   + \alpha \,  v \cdot \partial_x h - \beta \partial_x \phi(f) \cdot \partial_v h  =  g.
\end{equation}
Then we have
\begin{multline}
\label{cuisine}
\forall\, t > 0, \quad \Norm{h(t)}{\Hc^r_\nu}^2 \leq \Norm{h(0)}{\Hc^r_\nu}^2 + C \int_0^t ( \alpha + \beta \Norm{f(\sigma)}{\Hc^r_\nu}) \Norm{h(\sigma)}{\Hc^r_\nu}^2 \dd \sigma \\
+ 2 \int_0^t \Norm{g(\sigma)}{\Hc^r_\nu}\Norm{h(\sigma)}{\Hc^r_\nu} \dd \sigma.
\end{multline}

\end{lemma}

\begin{proof}
Let $\mathcal{L}_{\kappa}$ be the operator $\Lc_\kappa h = \{Ê\kappa,h\}$ with
$$
\kappa(x,v) = \frac{\alpha}{2}|v|^2  + \beta\phi(f)(x).
$$
The equation \eqref{tuyauxpvc} is thus equivalent to $\partial_t h - \{\kappa,h\} = g$.
Let $D$ be a linear operator.
We calculate that
\begin{eqnarray*}
\frac{\dd}{\dd t} \Norm{Dh}{L^2}^2 &=& 2 \langle Dh, D \Lc_\kappa h \rangle_{L^2} + 2 \langle Dh, D g \rangle_{L^2}\\
&=& 2\langle Dh,  \Lc_\kappa D h \rangle_{L^2} + 2\langle Dh, [D, \Lc_\kappa] h \rangle_{L^2}+ 2 \langle Dh, D g \rangle_{L^2},
\end{eqnarray*}
where $[D, \Lc_\kappa] = D \Lc_\kappa - \Lc_\kappa D$ is the commutator between the two operators.
The first term in the previous equality can be written
$$
2\langle Dh,  \Lc_\kappa D h \rangle_{L^2} = \langle 1,\Lc_\kappa(Dh)^2 \rangle_{L^2} =  \langle  \Lc_\kappa^*1,(Dh)^2 \rangle_{L^2} = 0,
$$
where $\Lc_\kappa^*$ is the $L^2$ adjoint of $\Lc_\kappa$, upon using the fact that Hamiltonian vector fields are divergence free.
Hence we get
$$
\frac{\dd}{\dd t} \Norm{Dh}{L^2}^2 = 2\langle Dh, [D, \Lc_\kappa] h \rangle_{L^2}+ 2 \langle Dh, D g \rangle_{L^2}.
$$
Now we consider the operators   $D = D^{m,p,q} = v^m \partial_x^p \partial_v^q$ for given muti-indices $(m,p,q) \in \mathbb{N}^{3d}$ such that $|m|  \leq \nu$ and $|p| + |q| \leq r$.
It is then clear that the second term in the right-hand side can be bounded by $2 \Norm{h}{\Hc_\nu^r} \Norm{g}{\Hc_\nu^r}$, and we are led to prove that
\begin{equation}
\label{irvi}
|Ê\langle Dh, [D, \Lc_\kappa] h \rangle_{L^2}|Ê\leq C (\alpha + \beta \Norm{f}{\Hc_{\nu}^r})\Norm{h}{\Hc_{\nu}^r}^2.
\end{equation}
The operator $\Lc_\kappa$ can be split into a linear combination of operators of the form $\Lc_v^i  = v_i  \partial_{x_i} $ and $\Lc_\phi^i =  - \partial_{x_i} \phi(f)(x)   \partial_{v_i}$ for $i = 1,\ldots, d$. We compute that for any smooth function $h$
\begin{eqnarray*}
[ÊD^{m,p,q}, \Lc_v^i ] hÊ&=&  v^m  \partial_x^p \partial_v^q (v_i  \partial_{x_i} h) - (v_i  \partial_{x_i}) v^m  \partial_x^p \partial_v^q h\\
Ê&=&  v^m   \partial_{x_i} \partial_x^p \partial_v^q (v_i h) - v_i v^m \partial_{x_i}  \partial_x^p \partial_v^q h\\
Ê&=&  q_i   v^m   \partial_{x_i} \partial_x^p \partial_v^{q - \langle i \rangle} h + v^m \partial_{x_i} \partial_x^p v_i \partial_v^q h  - v_i v^m \partial_{x_i}  \partial_x^p \partial_v^q h\\
&=&  q_i D^{m,p + \langle i \rangle,  q - \langle i \rangle} h,
\end{eqnarray*}
where $\langle i\rangle $ is the multi-index with coefficients  $\delta_{ij}$ the Kronecker symbol, for $j = 1,\ldots, d$ (we  make the convention that $D^{m,p,q} = 0$ when $p$ or $q$ contains negative index). Hence, as $| p  + \langle i \rangle| + |q- \langle i \rangle| \leq r$ as soon as $|q- \langle i \rangle| \geq 0$ we get
$$
\Norm{[D^{m,p,q},\Lc_v^i] h}{L^2} \leq C \Norm{h}{\Hc^r_\nu},
$$
where the constant $C$ depends on $r$.
This gives the first term in the right-hand side of \eqref{irvi}.

For the second term, we compute
\begin{eqnarray}
[ÊD^{m,p,q}, \Lc_\phi^i   ]hÊ&=&  - v^m \partial_x^p \partial_v^q ( \partial_{x_i} \phi \partial_{v_i}h) + (\partial_{x_i} \phi \partial_{v_i})v^m \partial_x^p \partial_v^q h\nonumber \\
&=&  m_i (\partial_{x_i} \phi )v^{m - \langle i\rangle } \partial_x^p \partial_v^q h - \sum_{k \neq 0} \Big(\begin{matrix} p \\ k \end{matrix}\Big)( \partial_x^k \partial_{x_i} \phi) v^m \partial_x^{p-k} \partial_v^q \partial_{v_i} h, \label{paugam}
\end{eqnarray}
(with the usual convention that $v^m = 0$ if $m$ contains a negative index).
The first term is easily bounded: we have
$$
\Norm{(\partial_{x_i} \phi )v^{m - \langle i\rangle } \partial_x^p \partial_v^q h}{L^2} \leq C \Norm{\partial_{x_i} \phi}{L^\infty_x}\Norm{h}{\Hc^{r}_\nu},
$$
and using \eqref{bach} with $p = \langle i \rangle$, we obtain the estimate, as $r \geq \nu - 1$.

In the second term, when $|k| +  \nu - 1  \leq r$ we can estimate directly
\begin{eqnarray*}
\Norm{( \partial_x^k \partial_{x_i} \phi) v^m \partial_x^{p-k} \partial_v^q \partial_{v_i}h }{L^2} &\leq& \Norm{\partial_x^k \partial_{x_i} \phi(f)}{L^{\infty}_x}  \Norm{h}{\Hc^r_\nu} \\
&\leq & C\Norm{f}{\Hc_\nu^{|k| + 1 + \nu - 2}}\Norm{h}{\Hc^r_\nu},
\end{eqnarray*}
after using \eqref{bach}, which gives the desired bound.

When $|k| \geq r + 1 -  \nu$ we can estimate
\begin{eqnarray*}
\Norm{( \partial_x^k \partial_{x_i} \phi) v^m \partial_x^{p-k} \partial_v^q \partial_{v_i}h }{L^2} &\leq& \Norm{ \partial_x^k \partial_{x_i} \phi}{L^2_x} \Norm{v^m \partial_x^{p-k} \partial_v^q \partial_{v_i}h }{L^\infty}\\
&\leq&C \Norm{f}{\Hc_\nu^{|k|Ê-1 }}  \Norm{v^m \partial_x^{p-k} \partial_v^q \partial_{v_i}h }{H^{2\nu}}\\
&\leq& C\Norm{f}{\Hc_\nu^{|k|Ê-1 }} \Norm{h}{\Hc_\nu^{2 \nu + |p-k| + |q|Ê+ 1}},
\end{eqnarray*}
by using \eqref{alouette} and Gagliardo-Nirenberg inequality in $\T^dÊ\times \R^d$.
Now in the sum of the second term in \eqref{paugam} we have $|k| \leq |p| \leq r$ otherwise the term is zero.
We thus have
$|p-k| + |q| \leq |p| - |k| + |q|$. As $|p| + |q| \leq r$,  then under the condition $|k| \geq r + 1 -  \nu$ considered here, we have
$|p-k| + |q| \leq \nu - 1$. We thus get the result, provided $3 \nu < r$, in order to bound both terms in the previous equation with the help of  $\Hc^r_\nu$ norms.
\end{proof}

\begin{proof}[Proof of Theorem \ref{th63appendix}]
We define the sequence of function $(f_n(t,x,v))_{n \in \N}$ as follows: for $t \in [0,T]$, $f_0(t,x,v) = f_0(x,v)$, and for $n \geq 0$, given $f_n \in \Hc^{r+2\nu + 1}_\nu$, we set $f_{n+1}(t,x,v)$ the solution of
\begin{equation}
\label{eq:rec}
\partial_t f_{n+1} + \alpha v \cdot \partial_x f_{n+1} - \beta \partial_x \phi(f_n) \cdot \partial_v f_{n+1} = 0,
\end{equation}
with initial value $f_{n+1}(0,x,v) = f_0(x,v)$.
Let $h_n$ be the Hamiltonian
$$
h_n (x,v) =  \frac{\alpha}{2} {|v|}{}^2 + \beta \phi(f_n)(x),
$$
and $\chi^t_n(x,v)$ its flow. Note that as $f_n$ is in $\Hc^{r+2\nu + 1}_\nu$ the flow of the microcanonical Hamiltonian $h_n$ is well defined. Moreover, this flow exists globally in time, since $\phi(f_n)(x)$ is bounded as $x \in \T^d$ a compact domain. The function $f_{n+1}(t,x,v)$ is thus well defined using characteristics: $f_{n+1}(t,x,v) = f_0( \chi^{-t}_n(x,v))$.

Let us apply  Lemma \ref{lem:1} with the space $\Hc^{r+2\nu + 1}_\nu$. We get for all $t > 0$,
\begin{multline*}
\Norm{f_{n+1}(t)}{\Hc^{r+2\nu + 1}_\nu}^2 \leq \Norm{f_0}{\Hc^{r+2\nu + 1}_\nu}^2 \\ + C \int_0^t ( \alpha + \beta \Norm{f_n(\sigma)}{\Hc^{r+2\nu + 1}_\nu}) \Norm{f_{n+1}(\sigma)}{\Hc^{r+2\nu + 1}_\nu}^2 \dd \sigma.
\end{multline*}
This shows by induction that we have
$$
\forall\, n,\quad\forall\, t \in [0,T], \quad  \Norm{f_n(t)}{\Hc^{r+2\nu + 1}_\nu} \leq 2 \Norm{f_0}{\Hc^{r+2\nu + 1}_\nu} \leq 2B,
$$
provided $T$ is small enough, namely $CT (\alpha + 2\beta \Norm{f_0}{\Hc^{r+2\nu + 1}_\nu}) < 1/4$ (which is implied if we assume \eqref{eq:T} for a suitable constant $C_{r,\nu}$). An application of Gronwall's lemma then implies \eqref{karlsruhe}.

Now  we can write that
\begin{multline*}
\partial_t(f_{n+1} - f_n) + \alpha v \cdot \partial_x (f_{n+1}  - f_n) - \beta  \partial_x \phi(f_n) \cdot \partial_v (f_{n+1} - f_n)\\
= \beta \partial_{x} \phi (f_{n} - f_{n-1}) \cdot \partial_v f_n.
\end{multline*}
Using Lemma \ref{lem:1} with the space $\Hc^r_\nu$, there exists a constant $C$ such that
\begin{multline*}
\Norm{f_{n+1}(t) - f_n(t)}{\Hc^r_\nu}^2 \leq C( \alpha + 2 \beta B) \int_0^t  \Norm{f_{n+1}(\sigma) - f_n(\sigma)}{\Hc^r_\nu}^2 \dd \sigma \\
+ 2 \int_0^t  \Norm{f_{n+1}(\sigma) - f_n(\sigma)}{\Hc^r_\nu} \Norm{\partial_{x} \phi (f_{n} - f_{n-1}) \cdot \partial_v f_n}{\Hc^r_\nu}  \dd \sigma.
\end{multline*}
Now for an operator of the form $D = v^m \partial_x^p \partial_v^q$ with $|p| + |q| \leq r$ and $|m| \leq \nu$, we have
\begin{align*}
\Norm{D (\partial_{x} \phi &(f_{n} - f_{n-1}) \cdot \partial_v f_n) }{L^2} \\
&\leq
C \Norm{ \phi (f_{n} - f_{n-1})}{H^{r+1}_x} \sum_{|a| + |b| \leq r+1}\Norm{  v^m \partial_x^a \partial_v^b f_n}{L^\infty}\\
&\leq C \Norm{f_n - f_{n-1}}{\Hc^r_\nu} \Norm{f_n}{\Hc^{r + 2\nu + 1}_\nu} ,
\end{align*}
using \eqref{alouette} and Gagliardo-Nirenberg inequality in $\T^d \times \R^d$.
Hence, using the previous uniform bound on $\Norm{f_n}{\Hc^{r + 2\nu + 1}_\nu} $, we obtain
\begin{multline*}
\Norm{f_{n+1}(t) - f_n(t)}{\Hc^r_\nu}^2 \leq C(1+B) \int_0^t  \Norm{f_{n+1}(\sigma) - f_n(\sigma)}{\Hc^r_\nu}^2 \dd \sigma \\
+ C (1+B)\int_0^t  \Norm{f_{n+1}(\sigma) - f_n(\sigma)}{\Hc^r_\nu} \Norm{f_{n}(\sigma) - f_{n-1}(\sigma)}{\Hc^r_\nu}  \dd \sigma.
\end{multline*}
From this estimate we deduce by a Gronwall inequality that
$$
\Norm{f_{n+1}(t) - f_n(t)}{\Hc^r_\nu}^2 \leq e^{3(1+B)T/2} (1+B)/2 \int_0^T\Norm{f_n(\sigma) - f_{n-1}(\sigma)}{\Hc^r_\nu}^2 \dd \sigma.
$$
For $T(1+B)$ - see estimate \eqref{eq:T} -  sufficiently small, this shows that
$$
\sup_{t \in [0,T]} \Norm{f_{n+1}(t) - f_n(t)}{\Hc^r_\nu}  \leq \frac{1}{2} \sup_{t \in [0,T]} \Norm{f_{n}(t) - f_{n-1}(t)}{\Hc^r_\nu}.
$$
We deduce that the sequence $f_n$ converges in $\Hc_\nu^r$, uniformly in time. The limit is then a solution of the Vlasov--Poisson equation in $C^1([0,T], \Hc_\nu^r)$.

Now, if we take two solutions, we have as before
$$
\partial_t(f - g) = -\alpha v \cdot \partial_x (f - g) + \beta  \partial_x \phi(f) \cdot \partial_v (f - g)
+ \beta \partial_{x} \phi (f-  g) \cdot \partial_v g.
$$
Using the previous lemma and the estimates we have on $f$ and $g$ on the interval $[0,T]$, we get as before
\begin{multline*}
\Norm{f(t) - g(t)}{\Hc^r_\nu}^2 \leq \Norm{f(0) - g(0)}{\Hc_\nu^r} \\+ e^{3(1+B)T/2} (1+B)/2 \int_0^t\Norm{f(\sigma) - g(\sigma)}{\Hc^r_\nu}^2 \dd \sigma,
\end{multline*}
from which we easily deduce the second estimate \eqref{eqlip} (using \eqref{eq:T}).
\end{proof}

\subsection{Convergence of splitting methods}

\subsubsection{Classical splitting methods}

We can now prove the following convergence result:
\begin{theorem}
\label{th23}
Let $\psi_p^\tau$ a splitting method of the form \eqref{dhom.1} fulfilling the condition of Definition \ref{gortoz} for some number $p$ and sufficiently large indices $\nu_0,\nu_p, r_0,r_p$.
Then it is convergent in the following sense:
For given $\nu > \nu_0$ and $r > r_0$, there exists $C_*$ such that for  $f \in \Hc^{r+ r_p}_{\nu + \nu_p}$,   $\varphi_{\Hc}^t(f)$ exists for  $t \in [0,T]$ with
$T = C_{*} (1+B)^{-1}$ with $B = \Norm{f}{\Hc^{r+ r_p}_{\nu + \nu_p}}$,  and  there exist $\tau_0$ and $C$ such that for all $\tau \leq \tau_0$, we have
\begin{equation}
\label{convergence}
\Norm{(\psi_p^\tau)^n (f)  - \varphi_{\Hc}^t(f)}{\Hc^r_\nu} \leq C \tau^p,
\end{equation}
for $t=n \tau \leq T$.
\end{theorem}
\begin{proof}
We can assume that $r_0$ and $\nu_0$ (and hence $r$ ans $s$) are large enough to ensure that the solution $\varphi_{\Hc}^t(f)$ exists over the time interval $T = C_{r,\nu} (1+B)^{-1}$ given by Theorem \ref{th63appendix} with $\alpha = \beta = 1$, see \eqref{eq:T}.

Now let us consider a splitting method of the form \eqref{dhom.1} with coefficients $(a_{i}, b_i)$. Using Theorem \ref{th63appendix} with the indices $(r + r_p - 2 \nu - 1, \nu + \nu_p)$, we can apply  \eqref{karlsruhe} alternately to $\varphi_{\Tc}^{a_i \tau}$ (i.e. $\beta = 0$) and $\varphi_{\Uc}^{b_i \tau}$ (i.e. $\alpha = 0$), and we obtain
$$
\Norm{\psi_p^\tau(f)}{\Hc^{r+r_p}_{\nu + \nu_p}} \leq e^{L (1+B) \tau} \Norm{f}{\Hc^{r+r_p}_{\nu + \nu_p}},
$$
where
$$
L = L_{r + r_p - 2 \nu -1, \nu + \nu_p} \Big( \sum_{i = 1}^s |a_i| + \sum_{i = 1}^{s+1} |b_i| \Big).
$$
This implies the bound
$$
\forall\, n \tau \leq T, \quad
\Norm{(\psi_p^\tau)^n(f)}{\Hc^{r+r_p}_{\nu + \nu_p}} \leq e^{L (1+B) T} B = e^\kappa B,
$$
where $\kappa$ is a factor depending on $r,r_p,\nu$ and $\nu_p$, but not on $B$.
Let us use the notation $f(t) = \varphi_{\Hc}^t$ and $f_n = (\psi_p^\tau)^n(f)$. We can write
\begin{equation}
\label{ibanez}
\Norm{f_{n+1} - f(t_{n+1})}{\Hc_\nu^r} \leq \Norm{\psi_p^\tau(f_{n}) - \varphi^\tau_{\Hc}(f_{n})}{\Hc_\nu^r}
+ \Norm{\varphi^\tau_{\Hc} (f_n)  - \varphi^\tau_{\Hc} (f(t_n))}{\Hc_\nu^r}.
\end{equation}
By applying again Theorem \eqref{th63appendix} with the constant $e^\kappa B$ instead of $B$, we can define the flow $\varphi_{\Hc}^t$ over a time interval of the form $C_*(1+B)^{-1}$ by possibly adapting $C_*$ (this is due to the fact that $\kappa$ does not depend on $B$), and such that for all $f$ and $g$ satisfying $\Norm{f}{\Hc_{\nu + \nu_p}^{r + r_p}} \leq e^{\kappa}B$, and $\Norm{f}{\Hc_{\nu + \nu_0}^{r + r_p}} \leq e^{\kappa}B$, we have
$$
\Norm{\varphi_{\Hc}^\tau(f) - \varphi_{\Hc}^\tau(g)}{\Hc_\nu^r} \leq e^{L_*(1+B)\tau } \Norm{f - g}{\Hc_\nu^r},
$$
for some constant $L_*$ depending on $r,\nu, s_p$ and $\nu_p$.

From \eqref{ibanez}, we obtain
$$
\Norm{f_{n+1} - f(t_{n+1})}{\Hc_\nu^r} \leq \Norm{\psi_p^\tau(f_{n}) - \varphi^\tau_{\Hc}(f_{n})}{\Hc_\nu^r}
+ e^{L_*(1+B) \tau}\Norm{f_n  - f(t_n)}{\Hc_\nu^r}.
$$
Using now the definition \ref{gortoz} of the order $p$ of the method, applied with the bounded set $\mathcal{B}$ defined as the ball of radius $e^\kappa B$ in the space $\Hc^{\nu + \nu_0}_{r + r_p}$, we obtain
$$
\Norm{f_{n+1} - f(t_{n+1})}{\Hc_\nu^r} \leq C \tau^{p+1}
+ e^{L_*(1+B) \tau}\Norm{f_n  - f(t_n)}{\Hc_\nu^r},
$$
which gives the result by induction.
\end{proof}

\subsubsection{Splitting methods with iterated commutators}

To end this section, we give arguments to show that Theorem \ref{th23} still holds for more general splitting methods defined by formulas \eqref{sche.1} \and \eqref{sche.2} using flows associated with iterated commutators.

Let us first consider the one-dimensional case $ d = 1$. In this case, the flow of
$[[\Tc, \Uc],\Uc] = 2 m \, \Uc$,
as well as the flows of the high-order commutators $W_{5,1}$, $W_{7,1}$ and $W_{7,2}$ (see \eqref{modifpotb}) are in fact flows of the Hamiltonian $\Uc$ scaled in time. Hence, all the previous convergence results extend straightforwardly to numerical schemes of the form \eqref{sche.2}.

In the higher dimensional cases $d \geq 2$, the flow $\varphi_{[[\Tc,\Uc],\Uc]}^t$ is given by the formulas \eqref{dorian}, \eqref{dorian2}. We see that it has the same structure as the flow $\varphi_{\Uc}^t$, but the potential $\phi(f)$ is replaced by the potential $K(f)$ given by the formula \eqref{gugliemi}. This potential satisfies the following estimates (compare with Lemma \ref{lem.51})
\begin{lemma}
Let $\nu>d/2$. Then we have for $r \geq 2 + d/2$
\begin{equation}
\label{alouette2}
\Norm{K(f)(x)}{H_x^{r}} \leq C \Big( \Norm{f}{\Hc^{r -2}_\nu} + \Norm{f}{\Hc^{r- 2}_\nu}^2 \Big)
\end{equation}
and
\begin{equation}
\label{bach2}
\Norm{ K(f)(x)}{W^{r,\infty}_x} \leq C \Big(\Norm{f}{\Hc^{r + \nu - 2}} + \Norm{f}{\Hc^{r + \nu - 2}_\nu}^2 \Big).
\end{equation}
\end{lemma}

\begin{proof}
Using the definition \eqref{eqonK_lap} of $K$, we have
$$
K = 2m  \phi - 2\sum_{i, j = 1}^d  (-\Delta_x)^{-1}\left( \partial_{x_i}\partial_{x_j}\phi  \right)^2 + 2(-\Delta_x)^{-1}(\Delta_x \phi)^2.
$$
We deduce that for $r \geq 2$,
\begin{equation*}
\begin{aligned}
\Norm{&K}{W^{r,\infty}_x}  \\
&\leq C \Big( \Norm{\phi}{W^{r,\infty}_x}Ê +  \sup_{i,j = 1, \dots, d} \Norm{\left( \partial_{x_i}\partial_{x_j}\phi  \right)^2}{W_x^{r-2,\infty}} +  \Norm{\left( \Delta_x \phi  \right)^2}{W_x^{r-2,\infty}}\Big)\\
&\leq  C \Big( \Norm{\phi}{W^{r,\infty}_x}Ê +   \sup_{i,j = 1, \dots, d} \Norm{ \partial_{x_i}\partial_{x_j}\phi }{W_x^{r-2,\infty}}^2+  \Norm{\Delta_x \phi  }{W_x^{r-2,\infty}}^2\Big),
\end{aligned}
\end{equation*}
and we deduce \eqref{alouette2} from the estimate \eqref{bach} for $\phi(f)$.
For the $L^2$ estimates, we have
$$
\Norm{K}{H^r_x}
\leq  C \Big( \Norm{\phi}{H^{r}_x}Ê +   \sup_{i,j = 1, \dots, d} \Norm{ ( \partial_{x_i}\partial_{x_j} \phi)^2 }{H_x^{r-2}}+  \Norm{(\Delta_x \phi)^2  }{H_x^{r-2}}\Big),
$$
and we conclude in a similar way by using the fact that for $\alpha > d/2$,
$$
\forall\,u,v \in H_x^\alpha, \quad
\Norm{u v }{H_x^{\alpha}} \leq C \Norm{u }{ H_x^{\alpha}} \Norm{ v }{ H_x^{\alpha} }
$$
for some constant $C$ depending only on $\alpha$ and $d$. We then deduce the result by applying the estimate \eqref{alouette} on $\phi(f)$.
\end{proof}
With these estimates in hand, it is easy to show that an existence result like Theorem
\ref{th63appendix} Êholds for the flow $\varphi_{[[\Tc,\Uc],\Uc]}^t$ with similar estimates.
Convergence results for splitting methods
\eqref{sche.1} can then be easily shown as in the proof of Theorem \ref{th23}.

\section{Conclusion}
In this work, new time splitting schemes are proposed for the Vlasov--Poisson system.
They are based on the decomposition of the Hamiltonian ${\Hc}$ between the kinetic ${\Tc}$ and
electric ${\Uc}$ part. In the one-dimensional case, the relation $[[{\mathcal T}, {\mathcal U}], {\mathcal U}] = 2m\, {\mathcal U}$
enables to design very efficient (with optimized number of flows)
high-order splitting using the modified potential approach. This can be generalized
to arbitrary dimension, the price to pay being to compute
the flow associated to the commutator $[[{\mathcal T}, {\mathcal U}], {\mathcal U}]$
which only depends on the spatial
variables; in this case also, new high-order splitting are proposed which turns out to be
very efficient compared to the existing splitting of the literature.
Finally, a convergence result of such splitting methods applied to the Vlasov--Poisson system
is obtained.




\begin{thebibliography}{10}

%
\bibitem{bedros}
\sc{B. Afeyan, F. Casas, N. Crouseilles, A. Dodhy, E. Faou, M. Mehrenberger, E. Sonnendr\"ucker},
\textit{Simulations of Kinetic Electrostatic Electron Nonlinear (KEEN) Waves with Two-Grid, Variable Velocity Resolution and High-Order Time-Splitting},
Eur. Phys. J. D {\bf 68}, p. 295 (2014).

%
\bibitem{BeMe2008}
\sc{N. Besse, M. Mehrenberger},
\textit{Convergence of classes of high-order semi-Lagrangian schemes for the Vlasov equation,}
Math. Comp. {\bf 77}, pp. 93-123 (2008).

%
\bibitem{sisc06}
{\sc S. Blanes, F. Casas, A. Murua},
\textit{Composition methods for differential equations with
processing},
SIAM. J. Sci. Comput. \textbf{27},  pp. 1817-1843 (2006).

%
\bibitem{blanes2008}
{\sc S. Blanes, F. Casas, A. Murua},
\textit{Splitting and composition methods in the numerical integration of differential equations},
Bol. Soc. Esp. Mat. Apl.  {\bf 45}, pp. 89-145 (2008).

%
\bibitem{BlaMo2000}
{\sc S. Blanes, P. C. Moan},
\textit{Practical Symplectic Partitioned Runge-Kutta and Runge-Kutta-Nystrom Methods},
J. Comput. Appl. Math. {\bf 142}, pp. 313-330 (2002).

%
\bibitem{ChaDeMe2011}
{\sc F. Charles, B. Despr\'es, M. Mehrenberger},
\textit{Enhanced  convergence estimates  for semi-Lagrangian
schemes. Application to the  Vlasov--Poisson equation},
SIAM J. Numer. Anal. {\bf 51}, pp. 840-863 (2014).

%
\bibitem{CroMeSo2010}
{\sc N. Crouseilles, M. Mehrenberger, E. Sonnendr\"ucker},
\textit{Conservative semi-Lagrangian schemes for Vlasov equations},
J. Comput. Phys.  {\bf 229}, pp. 1927-1953 (2010).

%
\bibitem{Degond86}
{\sc P. Degond},
\textit{Global existence of smooth solutions for Vlasov--Fokker--Planck equation in 1 and 2 space dimension},
Ann. Sci. \'Ec. Norm. Sup. {\bf 19}, pp. 519-542 (1986)

%
\bibitem{despres}
{\sc B. Despr\'es},
\textit{Uniform asymptotic stability of StrangÕs explicit compact schemes for linear advection},
SIAM J. Numer. Anal. {\bf 47}, pp. 3956-3976 (2009).

%
\bibitem{Lukas}
{\sc L. Einkemmer and A. Ostermann},
{\em Convergence analysis of Strang splitting for Vlasov--type equations},
SIAM J. Numer. Anal. {\bf 52}, pp. 140-155 (2014)

%
\bibitem{ffm}
{\sc M. Falcone, R. Ferretti, and T. Manfroni},
\textit{Optimal discretization steps in semi-
Lagrangian approximation of first order PDEs,}
in Numerical Methods for Viscosity Solutions and Applications (Heraklion, 1999),
Ser. Adv. Math. Appl. Sci. {\bf 59}, M. Falcone and C. Makridakis, eds.,
World Scientific, River Edge, NJ, pp. 95-117  (2001).

%
\bibitem{jcp14aho}
{\sc Y. G\"u\c{c}l\"u, A.J. Christlieb, and W.N.G. Hitchon},
\textit{Arbitrarily high-order convected scheme solution of the Vlasov--Poisson system},
J. Comput. Phys. \textbf{270}, pp. 711-752 (2014).

%
\bibitem{hairer}
{\sc E. Hairer, C. Lubich, G. Wanner},
\textit{Geometric numerical integration: Structure-Preserving Algorithms for Ordinary Differential Equations},
Springer Series in Computational Mathematics, 2006.

%
\bibitem{leimreich}
{\sc B. Leimkuhler and S. Reich},
{\em Simulating Hamiltonian dynamics},
Cambridge University Press, Cambridge Monographs on Applied and Computational Mathematics 14, 2004.

%
\bibitem{Marsden99}
{ \sc J. E. Marsden and T. S. Ratiu},
{\em Introduction to Mechanics and Symmetry},
Springer-Verlag, New York, 1999.

%
\bibitem{mclachlan02}
{\sc R.I. McLachlan},
\textit{Families of high-order composition methods},
Numer. Algorithms {\bf 31}, pp. 233-246 (2002).

%
\bibitem{mclachlan}
{\sc R.I. McLachlan, R. Quispel},
\textit{Splitting methods},
Acta Numer. {\bf 11}, pp. 341-434 (2002).

%
\bibitem{omelyan}
{\sc I.P. Omelyan, I.M. Mryglod, R. Folk},
\textit{Symplectic analytically integrable decomposition algorithms: classification, derivation, and application to
molecular dynamics, quantum and celestial mechanics simulations},
Comput. Phys. Comm. {\bf 151}, pp. 272-314 (2003).

%
\bibitem{Shoucri2005}
{\sc E. Pohn, M. Shoucri, G. Kamelander},
\textit{Eulerian Vlasov codes},
Comput. Phys. Comm. {\bf 166}, pp. 81-93 (2005).

%
\bibitem{ReSo2011}
{\sc T. Respaud, E. Sonnendr\"ucker},
\textit{Analysis of a new class of Forward semi-Lagrangian
schemes for the 1D Vlasov--Poisson Equations},
Numer. Math. {\bf 118}, pp. 329-366 (2011).

%
\bibitem{Schaeffer2010}
{\sc J. Schaeffer},
\textit{Higher order time splitting for the linear Vlasov equation},
SIAM J. Numer Anal. {\bf 47}, pp. 2203-2223 (2009).

%
\bibitem{Sofroniou}
{\sc M. Sofroniou, G. Spaletta},
\textit{Derivation of symmetric composition constants for symmetric integrators},
Optimization Methods \& Software {\bf 20}, pp. 597-613 (2005).

%
\bibitem{WatanSug2004}
{\sc T. H. Watanabe, H. Sugama},
\textit{Vlasov and Drift Kinetic Simulation Methods Based on the Symplectic Integrator},
Transp. Th. Stat. Phys. {\bf 34}, pp. 287-309 (2005).

%
\bibitem{Yoshida1990}
{\sc H. Yoshida},
\textit{Construction of higher order symplectic integrators},
Physics Letters A {\bf 150}, pp. 262-268 (1990).



\end{thebibliography}
\end{document}